\documentclass[12pt,reqno]{amsart}
\usepackage{}
\usepackage{a4wide}
\numberwithin{equation}{section}
\usepackage{mathrsfs}
\usepackage{amsfonts}
\usepackage{amsmath,extpfeil}
\usepackage{stmaryrd}
\usepackage{amssymb}
\usepackage{amsthm}
\usepackage{mathrsfs}
\usepackage{url}
\usepackage{amsfonts}
\usepackage{amscd}
\usepackage{indentfirst}
\usepackage{enumerate}
\usepackage{amsmath,amsfonts,amssymb,amsthm}
\usepackage{amsmath,amssymb,amsthm,amscd}
\usepackage{graphicx,mathrsfs}
\usepackage{appendix}
\usepackage[numbers,sort&compress]{natbib}
\usepackage{color}
\usepackage[colorlinks, linkcolor=blue, citecolor=blue]{hyperref}
\usepackage[notref,notcite,color,draft]{showkeys}

\makeatletter
\@namedef{subjclassname@2020}{\textup{2020} Mathematics Subject Classification}
\makeatother
\newcommand{\R}{\mathbb{R}}

\newcommand{\M}{\mathcal M}

\newtheorem{theorem}{Theorem}[section]
\newtheorem{lemma}{Lemma}[section]
\newtheorem{corollary}{Corollary}[section]
\newtheorem{proposition}{Proposition}[section]
\newtheorem{remark}{Remark}[section]    % Use {\rm ...}
\numberwithin{equation}{section}

\newcommand\vep{\eps }

\newcommand\eps{\varepsilon}
\newcommand{\p}{\partial}

\newcommand{\beq}{\begin{equation}}
\newcommand{\eeq}{\end{equation}}

\begin{document}

\title[$p$-Gauss curvature flow]
{Long time regularity of \\ the $p$-Gauss curvature flow with flat side}

\author[G. Huang, X.-J. Wang and Y. Zhou]
{Genggeng Huang, Xu-Jia Wang and Yang Zhou}

\address[Genggeng Huang]
{School of Mathematical Sciences, Fudan University, Shanghai 200433, China.}
\email{genggenghuang@fudan.edu.cn}

\address[Xu-Jia Wang]
{Mathematical Sciences Institute, The Australian National University, Canberra, ACT 2601, Australia}
\email{Xu-Jia.Wang@anu.edu.au}

\address[Yang Zhou]{School of Mathematics and Applied Statistics, University of Wollongong, Wollongong, NSW 2522, Australia}
\email{yang.zhou18798@outlook.com}

\thanks{The first author was supported by NNSFC 12141105,   
the second author was supported by ARC DP200101084.}


\date{}
\maketitle

\begin{abstract}
In this paper, we prove the long time regularity of the interface in the $p$-Gauss curvature flow with flat side
in all dimensions for $p>\frac1n$.
Here the interface is the boundary of the flat part in the flow.
In dimension $2$, this problem was solved in \cite{DL2004} for $p=1$ and in \cite{KimLeeRhee2013} for $p\in(1/2,1)$.
We utilize the duality method to transform the Gauss curvature flow
to a singular parabolic Monge-Amp\`ere equation,
and prove the regularity of the interface by studying the asymptotic cone of the parabolic Monge-Amp\`ere equation
in the polar coordinates.
\end{abstract}

\baselineskip16pt
\parskip5pt

\vspace{5mm}
\section{Introduction}\label{intro}
\vspace{3mm}

Let  $\mathcal M_0$ be a closed convex hypersurface in $\R^{n+1}$
whose position function is given by  $X_0(\omega)$,  $\omega\in \mathbb{S}^n$.
In this paper we study the following Gauss curvature flow with power $p$,
\begin{equation}\label{GCF-p}
\begin{split}
\frac{\partial X}{\partial t}(\omega,t)&=-K^p(\omega,t)\gamma(\omega,t),\\
X(\omega,0)&=X_0(\omega),
\end{split}
\end{equation}
where $K$ is the Gauss  curvature of $\mathcal M_t=X(\omega, t)$
and $\gamma$ is the outer unit normal of $\mathcal M_t$ at $X(\omega, t)$.

The Gauss curvature flow (with $p=1$) was first studied by Firey \cite{F1974},
as a model for the wear of stones under tidal waves.
Tso  \cite{Tso1985} proved that if $\M_0$ is strictly convex, then there is a unique smooth solution
$\mathcal M_t$ which shrinks to a point as $t\to V/\omega_n$,
where $V$ is the volume enclosed by $\mathcal M_0$ and $\omega_n$ is the surface area of
the unit sphere $\mathbb S^n$.
An open question was whether $\mathcal M_t$ shrinks to a round point when $t\to V/\omega_n$,
and it was confirmed by Andrews \cite{A1999} in dimension two.
In high dimensions, Andrews,
Guan, and  Ni \cite{AGN2016, GN2017} proved that the normalized Gauss curvature flow converges to
a self-similar solution,
and Brendle, Choi and Daskalopoulos \cite{BCD2017} proved that the self-similar solution must be a sphere.
In  \cite{AGN2016, BCD2017}  the results were also proved for the Gauss curvature flow \eqref{GCF-p} for $p>\frac{1}{n+2}$.

If the initial hypersurface $\M_0$ contains a flat side,
it was proven in \cite{Andrews2000,C1985} that the solution $\M_t$ becomes uniformly convex
and smooth instantly  for $t>0$ if $p\le \frac 1n$.
However,  if $p>\frac{1}{n}$,
Hamilton \cite{H1993} and Andrews  \cite{Andrews2000}
observed that the flat side does not instantly bend under the Gauss curvature flow and
it will persist for a while before the solution $\M_t$ becomes uniformly convex.
 In this case,
the $C^\infty$ regularity of the strictly convex part of $\M_t$ was proved in \cite{Tso1985, C1985},
and the strict convexity of $\M_t-F_t$ and the $C^{1,\alpha}$ regularity
across the interface $\Gamma_t$ were obtained in \cite{DS2009}.
Here $F_t\subset\M_t$ is the flat side, the interface $\Gamma_t$ is the boundary of  the flat side.

A particularly interesting question is the regularity of the interface $\Gamma_t$ for $p>\frac 1n$.
When $p=1$, Daskalopoulos and Hamilton \cite{DH1999}  proved the regularity of $\Gamma_t$ for small $t>0$
under certain conditions on the initial hypersurface $\mathcal M_0$.
If $n=2$, the regularity of $\Gamma_t$ was obtained by Daskalopoulos and Lee \cite{DL2004}
for all time $t$ before it disappears, and
the result was extended to $p\in (1/2,1]$ in \cite{KimLeeRhee2013}.
However,  the regularity of  $\Gamma_t$ for large time $t$ is still open  in dimension two for $p>1$
and in high dimensions for all $p>\frac 1n$.

The objective of the paper is to establish the regularity of  $\Gamma_t$ for large time $t$ in all dimensions,
for all $p>\frac 1n$. Recently, a related Monge-Amp\`ere obstacle problem was investigated in  \cite{HTW}.
In this paper,  we will use some techniques from \cite{HTW},
but the argument in \cite{HTW} does not apply to the parabolic case directly,
due to  the lack of concavity of equation \eqref{v} and the strong degeneracy of the equation \eqref{v} near the interface. 
It is worth mentioning that due to the lack of concavity,
the global regularity of the first boundary value problem for equation \eqref{v}
was not solved until very recently in \cite{ZZ}.

Choosing the coordinates properly, we may assume that $\M_t\subset\{y_{n+1}\ge 0\}$ and
the flat side lies on the plane $\{y_{n+1}=0\}$.
For simplicity we assume that $\M_0$ has only one flat part.
Our argument also applies to the case when $\M_0$ has multiple flat parts
as long as they are strictly separate.
Then, locally $\M_t$  can be represented as the graph of a nonnegative function $v$,
\begin{equation*}
y_{n+1}=v(y_1, \cdots, y_n,t)
\end{equation*}
over a bounded domain $\Omega_t$, and $v$ satisfies the equation
\begin{equation}\label{v}
v_t=\frac{(\det D^2 v)^p}{(1+|D v|^2)^{\frac {(n+2)p-1}{2}}}.
\end{equation}
As $\M_t$ is a closed convex hypersurface,
we may assume that $|Dv(y,t)|\to \infty $ as $y\to \partial \Omega_t$.

For the short time existence of a solution with smooth interface,
it is necessary to assume  certain non-degeneracy conditions on the initial hypersurface $\M_0$  \cite{DH1999}.
Denote
\beq\label{def-g}
g=\big(\frac{\sigma_p+1}{\sigma_p} v\big)^{\frac{\sigma_p}{\sigma_p+1}},\ \ \ \sigma_p=n-\frac1p.
\eeq
 
According to \cite{DH1999}, see also  \cite{DL2004, D2014},
we assume the following non-degeneracy conditions.

\begin{itemize}
\item[(I1)] The level set $\{v(y, 0) = \eps \}$ is uniformly  convex for $\eps\ge 0$ small,
i.e., its principal curvatures have positive upper and lower bounds.
\item[(I2)] There exist a constant $\lambda_0\in (0, 1)$
such that $\lambda_0\le |Dg(y,0)|\le \lambda_0^{-1} \ \text{on} \ \Gamma_0$.
\end{itemize}

\noindent
Note that condition (I2) implies that  $v(y,0)\approx \text{dist}(y, \Gamma_0)^{(\sigma_p+1)/\sigma_p}$.
We also assume

\begin{itemize}
\item[(I3)]  $\M_0$ is locally uniformly convex and smooth away from the flat region, and 
$g(y,0)\in C_\mu^{2+\alpha}(\overline{\{v>0\}})$,  where $C_\mu^{2+\alpha}$ will be introduce in \eqref{1.15} below.

\end{itemize}

We have the following regularity and convexity results for the interface $\Gamma_t$.

\begin{theorem}\label{thmA}
Assume conditions {\rm (I1)-(I3)}.
Then if $p>\frac 1n$,  the interface $\Gamma_t$ is  smooth and uniformly convex $\forall\ t\in(0,T^*)$,
where $T^*>0$ is the time when the flat region disappears.
\end{theorem}

Through the investigation of the regularity of the interface $\Gamma_t$, we have also obtained the regularity of the function $g$ near the interface.  See Remark \ref{rem6.2}.

To prove Theorem \ref{thmA},
let $u(\cdot, t)$ be the Legendre transformation of $v(\cdot, t)$, i.e.
\begin{equation*}
u(x,t)=\sup\{y\cdot x-v(y,t)\ |\ y\in\Omega_t\}, \ \   x\in D_yv(\Omega_t)=\R^n.
\end{equation*}
Then $u(x,t)$ solves
\begin{equation}\label{u}
\det D^2 u=\frac{1}{(-u_t)^{\frac 1p}(1+|x|^2)^{\frac{(n+2)p-1}{2p}}}+c_t\delta_0,
\end{equation}
where $c_t$ is the volume of the flat part.  Hence $c_t>0$ for $t\in [0,T^*)$.
Without loss of generality,  we assume that the origin is an interior point of the convex set $\{v(\cdot, t)=0\}$ for all $t\in[0,T^*)$.
Then for any given $T\in (0, T^*)$, there is a positive constant $\rho_0$ such that
\begin{equation}\label{rho-0}
B_{\rho_0}(0)\subset\subset  \{y\in\mathbb R^n \ |\  v(y,t)=0\}, \ \   \forall~t\in[0,T].
\end{equation}
It implies that $u(0,t)=0$ and  $u(x,t)>\rho_0 |x|\ \forall\ x\ne 0$.

We first prove that the interface $\Gamma_t$ moves at finite speed,
namely $u_t$ satisfies the linear growth condition
\begin{equation}\label{ut-b}
C^{-1} |x|\le -u_{t}(x,t)\le C |x|,
\end{equation}
for $x\ne 0$ near the origin, where $C>0$ is a positive constant.
We then use \eqref {ut-b} to prove the key growth estimates
\begin{equation}\label{wx}
C^{-1}|x|^{n+1-1/p}\le w(x,t)\le C |x|^{n+1-1/p} ,
\end{equation}
near $x=0$,
where
$$w:=u-\phi$$
and $\phi(x,t)$ is the tangential cone of $u$ at $(0,t)$.
Unlike the elliptic case, where similar estimates can be obtained by
applying Pogorelov's technique to $v$ and its Legendre transform $u$.
In the parabolic case, we can apply Pogorelov's technique to equation \eqref{u},
but not to equation \eqref{v}, due to the lack of concavity.
This is the main difficulty in proving the regularity of the interface.

Fortunately, we found the following auxiliary function,
$$G=:\frac{x_ix_j u_{x_ix_j}}{(-u_t)^{\beta-4(p\beta+1)u_t}} \ \   \text{for} \ \   \beta\in (1,\sigma_p+1).$$
By careful computation, we obtain an upper bound for $G$,
which enables us to prove the key estimates \eqref{wx}.
From \eqref{wx} we obtain the $C^{1,1}$ regularity for $u$ in the polar coordinates.
The estimates \eqref{wx} also imply that the non-degeneracy conditions (I1)-(I2) hold for all time $t\in [0, T^*)$.

Express $u(\cdot, t)$ in the the spherical coordinates $(\theta, r)$.
Then the uniform convexity and smoothness of the interface $\Gamma_t$
is equivalent to the uniform convexity and the smoothness of the asymptotic cone $\phi$ \cite{HTW}.
The uniform convexity of $\phi$ is given in Corollary \ref{C3.4}.
For the smoothness of $\phi$, we introduce the function
\begin{equation*}\label{Trans2}
   \zeta(\theta, s, t) =\frac{u(\theta, r, t)}{r},\ \ \
   s =r^{\frac{\sigma_p}2},
\end{equation*}
where $(\theta, r)$ is the spherical coordinates for $x$.
Then the smoothness of  $\phi$ is equivalent to that of $\zeta$ on the boundary $\{s=0\}$.
We will prove the regularity of $\zeta$ in Theorem \ref{thm4}.
Therefore Theorem \ref{thmA} follows.
 
The function $\zeta$ satisfies the  parabolic Monge-Amp\`ere type equation:
\begin{equation}\label{1-po2}
-\zeta_t\det \begin{pmatrix}
 \zeta_{ss}+\frac{2+\sigma_p}{\sigma_p}\frac{\zeta_s}{s}   &   \zeta_{s\theta_{1}}&\cdots& \zeta_{s\theta_{n-1}}\\[3pt]
 \zeta_{s\theta_1} & \zeta_{\theta_1\theta_1}+\zeta+\frac {\sigma_p}2 s\zeta_s & \cdots&\zeta_{\theta_1\theta_{n-1}}\\[3pt]
                    \cdots &\cdots &\cdots &\cdots  \\[3pt]
 \zeta_{s\theta_{n-1}} &\zeta_{\theta_1\theta_ {n-1}} & \cdots&\zeta_{\theta_{n-1}\theta_{n-1}}+\zeta+\frac {\sigma_p}2 s\zeta_s
\end{pmatrix}^p=\bar {F} (s),
\end{equation}
in $\{s>0\}$,
where
$$\bar {F} (s)= 4^p \sigma_p^{-2p} \big(1+s^{4/\sigma_p}\big)^{-\frac{(n+2)p-1}2} . $$
 Note that $\bar {F} $ is only H\"older continuous in general which is the obstacle for higher regularity of $\zeta$ in $s$.

By estimates \eqref{ut-b} and \eqref{wx},  
we infer that $\zeta\in C^{1,1}$,
and equation \eqref{1-po2} is uniformly parabolic.
We then use the techniques in \cite{HTW} to show  that  $\zeta\in C^2$,
namely $\zeta_{t}$ and $D^2_{\theta, r}\zeta$ are  continuous up to $\{s=0\}$.
By a weighted $W^{2,p}$ estimate for linear parabolic equations \cite{DP2020},
we conclude that $\zeta(\theta,s,t)\in C^{2+\alpha}(\mathbb S^n\times[0,1]\times (0,T])$.
Our notation $C^{1,1}, C^2$, and $ C^{2+\alpha}$ will be introduced below.

The paper is organized as follows.
In Section \ref{s2} we prove estimates  \eqref{ut-b}. We then prove \eqref{wx} in Section \ref{s3}.
Sections \ref{s4}, \ref{s5} and \ref{s6} are devoted to the higher regularity of $\zeta$.

\vspace{2mm}

{\bf Notation.}
Given two positive quantities $a$ and $b$, we denote
\beq
a\lesssim b
\eeq
if there is a constant $C>0$, depending only on $\mathcal M_0, n, p, T$,
such that $a \le C b$, where $T\in (0, T^*)$ is any given constant.
We also denote
\beq \label{aab}
a\approx b
\eeq
if  $a\lesssim b$ and $b\lesssim a$.
Given two convex domains $A$ and $B$ in $\R^{n+1}$,
we denote $A\sim B$ if there exist points
$x_0\in A$ and $y_0\in B$ such that $C^{-1}(A-x_0) \subset B-y_0 \subset C(A-x_0)$.

Let $\Omega$ be a domain in $\mathbb R^n$.
As usual, we define the norm $\|\cdot\|_{C^{k,\alpha}(\overline \Omega)}$ by
\beq \label{norm1}
\| U\|_{C^{k,\alpha}(\overline \Omega)}=\sup_{|\gamma|\le k}|D^\gamma U(x)|
  + \sup_{|\gamma| = k \atop x,y\in\Omega}\frac{|D^\gamma U(x)-D^\gamma U(y)|}{|x-y|^\alpha}  ,
\eeq
where $k\ge 0$ is an integer,  $\alpha\in (0, 1)$.

Denote the norm  $\|\cdot\|_{C^{k+\alpha,\frac{k+\alpha}{2}}_{x,t}(\overline Q)}$ for
the parabolic H\"older space by
\beq \label{norm2}
\| U\|_{C^{k+\alpha, \frac{k+\alpha}{2}}_{x,t}(\overline Q)}
  =\sup_{|\gamma|+2s\le k \atop (x,t)\in Q}|D_x^\gamma D^s_t U(x,t)| + \sup_{|\gamma|+2s
  = k \atop (x,t), (y,t')\in Q}\frac{|D_x^\gamma D^s_t U(x,t)-D_x^\gamma D^s_t U(y,t')|}{(|x-y|^2 +|t-t'|)^{\alpha/2}},
\eeq
where $Q$ is a domain in the space-time $\R^n\times\R^1$.

For simplicity we will abbreviate the notations as follows.

\begin{itemize}

\item For $k\ge 0$ and $\alpha\in (0, 1)$,
we will write $\|\cdot \|_{C^{k+\alpha, \frac{k+\alpha}{2}}_{x,t}(\overline Q)}$ as $\|\cdot \|_{C^{k+\alpha}(\overline Q)}$ for brevity.\\
Hence for a function $U$ which is independent of $t$, the $C^{k+\alpha}$ norm is given by \eqref{norm1},
and for a function $U$ which depends on $t$, the $C^{k+\alpha}$ norm is given by \eqref{norm2}.

\item
We denote by $\|\cdot\|_{C^{1,1}(\overline Q)}$ ($\|\cdot\|_{C^2 (\overline Q)}$, resp.)  the norms  of functions such that
$|D_x^\gamma D^s_t U|$ are bounded (continuous, resp.), $\forall\ |\gamma|+2s\le 2$.

\item We will use spherical coordinates $(\theta, r)$ in our argument below.
In this case, we use $\|\cdot\|_{C^{1,1}}$ ($\|\cdot\|_{C^2}$, resp.)  to denote the norms for the functions of which
$|D_{\theta, r}^\gamma D^s_t U|$ are bounded (continuous, resp.), $\forall\ |\gamma|+2s\le 2$.

\end{itemize}

To study the regularity of $\zeta$,
as in \cite{DH1999, DL2004} we introduce H\"older spaces
with respect to the metric $\mu$ in $ \R^{n-1}\times \mathbb R^+ \times \R$,
\begin{equation}
\mu[(x, t), (y, s)] = |x'-y'| + |\sqrt{x_{n}} - \sqrt{y_{n}}| + \sqrt{|t-s|}.
\end{equation}
Denote $\R^{n, +}=\R^{n-1}\times \R^+=\{x\in\R^n\ |\  x_n>0\}$.
Let $Q$  be a domain in $\mathbb R^{n,+} \times \R$.
We define  the norm
\begin{equation}
\|U\|_{C_\mu^\alpha({\overline Q})} = \sup_{p\in {Q}} |U(p) |+  \sup_{p_1,p_2\in{Q}} \frac{|U(p_1)-U(p_2)|}{\mu[p_1,p_2]^\alpha}.
\end{equation}
Let $\|\cdot \|_{C_\mu^{2+\alpha}({\overline Q})}$ be norm
\begin{equation}\label{1.15}
\begin{split}
\|U\|_{C_\mu^{2+\alpha}({\overline Q})} & = \|x_n U_{nn}\|_{C_\mu^\alpha({\overline Q})} +
{\Small\text{$\sum_{i=1}^{n-1}$}} \|\sqrt{x_n}U_{ni}\|_{C_\mu^\alpha({\overline Q})} + {\Small\text{$\sum_{i,j=1}^{n-1}$}} \|U_{ij}\|_{C_\mu^\alpha({\overline Q})} \\
& \ \   + {\Small\text{$\sum_{i=1}^{n}$}} \|U_{i}\|_{C_\mu^\alpha({\overline Q})} +  \|U_{t}\|_{C_\mu^\alpha({\overline Q})}+ \|U\|_{C_\mu^\alpha({\overline Q})}.
\end{split}
\end{equation}
For integer $k\ge 1$,
we denote the norm $\|\cdot\|_{C_\mu^{k,2+\alpha}({\overline Q})}$ by
\begin{equation*}
\|U\|_{C_\mu^{k,2+\alpha}({\overline Q})}=\sum_{|\gamma|+2s\le k}\|D^\gamma_xD^s_t U\|_{C_\mu^{2+\alpha}({\overline Q})}.
\end{equation*}

For  $p\in(1,\infty)$, we also need the weighted Sobolev spaces
$W^{1,1}_{p}(Q,d\nu)$ with the norm
\begin{equation*}
\|U\|_{W^{1,1}_{p}(Q,d\nu)} =\|U\|_{L^p_{\nu}({Q})} +   \|U_t\|_{L^p_{\nu}({Q})} 
+ \|DU\|_{L^p_{\nu}({Q})} \ \ \
\end{equation*}
and  $W^{2,1}_{p}({Q},d\nu)$ with the norm
\begin{equation*}
\|U\|_{ W^{2,1}_{p}({Q},d\nu)} =\|U\|_{L^p_{\nu}({Q})} + \|U_t\|_{L^p_{\nu}({Q})}
       +   \|DU\|_{L^p_{\nu}({Q})}    + \|D^2U\|_{L^p_{\nu}({Q})},
\end{equation*}
respectively, where
$
\|U\|_{L^p_{\nu}({Q})}=\Big(\int_{{Q}}|U(x,t)|^p x_n^b dxdt\Big)^{1/p}. 
$

\vspace{5mm}
\section{Estimates for the speed of the interface}\label{s2}
\vspace{3mm}

First we recall the short time existence and regularity in \cite{DH1999},
where Daskalopoulos and Hamilton proved the following.

\begin{proposition}[Theorem 9.1,\cite{DH1999}]\label{DH1999}
Assume the conditions {\rm (I1)-(I3)}.
Then, there exists a time $T_0>0$
such that \eqref{GCF-p} admits a solution $\mathcal M_t$ for $0< t\leq T_0$, and
at any given time $t\in (0, T_0]$, $\M_{t}$ satisfies the conditions {\rm (I1)-(I3)}.
\end{proposition}
\begin{remark}\label{rm1}
\rm{Proposition \ref{DH1999} is proved in} \rm{\cite{DH1999}} for $n=2$. The proof also holds for high dimension case $n\ge 3$. Moreover, for $0<t\le T_0$,  the proof also implies the following condition(see Theorem 9.2 in \cite{DH1999}):\par
(I4) $g_{ij}\tau_i g_j\in L^\infty(\{v>0\})$ where $\tau=(\tau_1,\cdots,\tau_n)$ is the tangent vector field of the level set of $g$, i.e.  $\tau\cdot \nabla g=0$. 
\\Therefore, choosing a sufficiently small $t_0>0$ as the initial time, we may assume that (I4) holds at $t=0.$
\end{remark}

The long time regularity of the solution $\M_t$ was studied in \cite{DS2009},
from which we quote the following results.

\begin{proposition}\label{DS2009}
Let $\mathcal M_t$ be a solution to \eqref{GCF-p}.
Then  $\M_t- F_t$ is locally uniformly convex and smooth for any $t\in (0, T^*)$,
and $\mathcal M_t\in C^{1,\alpha}$ for some $\alpha>0$ as long as $\mathcal M_t$ exists.
\end{proposition}

We refer the reader to  \cite[Corollary 5.4]{DS2009} and  \cite[Theorem 8.4]{DS2009}
for the above results.
Proposition \ref{DS2009} implies that $v(y, t)$ is $C^{1,\alpha}$ near the interface $\Gamma_t$.

We first derive some estimates at time $t=0$. By Remark 2.1, we may assume conditions (I1)-(I4) hold at $t = 0.$
\begin{lemma}\label{urr-g-3}
Assume the conditions {\rm (I1)-(I4)}.
Then, for $r=|x|>0$ sufficiently small, there hold the estimates
\begin{equation}\label{3-ut}
\frac{1}{\bar {C}}  |x|\le -u_t(x,0)\le  \bar {C} |x| ,
\end{equation}
\begin{equation}\label{3-urr}
 u_{rr}(x, 0) \le \bar C r^{\sigma_p-1} = \bar C |x|^{n-1-1/p} ,
\end{equation}
where
\begin{equation}\label{u-rr}
u_{rr}=\frac{x_ix_ju_{ij}}{r^2}, \ \ \ r=|x|,
\end{equation}
and $\bar C$ is a positive constant depending on $ n, p, g(\cdot, 0)$.
\end{lemma}

\begin{proof}
By the non-degeneracy conditions (I1)-(I2),  there exists a constant $\bar \lambda_0$ such that
\beq \label{nondc}
{\begin{split}
 \bar \lambda_0 & \le  |Dg(\cdot, 0)| \le \bar \lambda_0^{-1},\\
 \bar \lambda_0 & \le  \lambda_{\eps,  i} \le \bar \lambda_0^{-1},
 \end{split}}
 \eeq
where $\lambda_{\eps,  i} \ (i=1, \cdots, n-1)$ are the principal curvatures of the level set $\{v(y, 0)=\eps\}$,
for  $\eps\ge 0$ small.

By the definition of $g$ in \eqref{def-g}, we compute
\begin{equation}\label{3-g-2}
{\begin{split}
v_t&=g^{\frac{1}{\sigma_p}} g_t,\\
v_i &= g^{\frac{1}{\sigma_p}} g_i,  \\
v_{ij} &= g^{\frac{1}{\sigma_p}} g_{ij} + \frac{1}{\sigma_p}g^{\frac{1}{\sigma_p}-1} g_{i}g_j.
\end{split}}
\end{equation}
Hence by equation \eqref{v}, $g$ satisfies
\begin{equation}\label{g-eq}
g_t=\frac{\left(g \det \big(D^2 g + \frac1{\sigma_p}g^{-1} Dg\otimes Dg\big)\right)^p}{\big(1+g^{\frac2{\sigma_p}}|D g|^2\big)^{\frac {(n+2)p-1}{2}}}.
\end{equation}
Here $Dg\otimes Dg$ is a matrix with $(i,j)$-entries $g_ig_j$.

Let ${\xi}^{(1)},\cdots,{\xi}^{(n-1)}, \xi^{(n)}=\frac{D g}{|D g|}$ be an orthonormal frame
at a given point $y_0\in \{v(\cdot,0)>0\}$ near the interface $\Gamma_{0}$.
Denote $x_0=D v(y_0, 0)$ and $r=|x_0|$.
By the duality of $u$ and $v$, we have
$$\xi^{(n)}=\frac{D v(y_0, 0)}{|D v(y_0, 0)|} = \frac{x_0}{|x_0|} .$$

At the point $(y_0, 0)$, by the non-degeneracy conditions and \eqref{3-g-2}, we have
$$ r=|Dv| = g^{\frac{1}{\sigma_p}} |Dg| \approx g^{\frac{1}{\sigma_p}}.$$
By  \eqref{nondc},
the eigenvalues of matrix $(g_{{\xi}^{(k)}{\xi}^{(l)}})_{k,l=1}^{n-1}$ fall in the interval
$(\bar \lambda_0, \bar \lambda_0^{-1})$.
By \eqref{3-g-2},
\beq\label{vxx}
 (v_{{\xi}^{(k)}{\xi}^{(l)}})_{k,l=1}^{n-1}
\approx g^{\frac{1}{\sigma_p}} (g_{{\xi}^{(k)}{\xi}^{(l)}})_{k,l=1}^{n-1}
\approx g^{\frac{1}{\sigma_p}}{\rm I}_{(n-1)\times (n-1)}  \approx r {\rm I}_{(n-1)\times (n-1)}.
\eeq
By (I3)-(I4), one knows  $g(y,0)\in  C^{2+\alpha}_\mu (\overline{\{v>0\}})$ and  $g_{\xi^{(i)}\xi^{(n)}}\in L^\infty$,  $i\ne n$, i.e. 
\begin{equation*}
gg_{\xi^{(n)}\xi^{(n)}}=o(1)
,\quad \sqrt g g_{\xi^{(i)}\xi^{(n)}}=o(1),\quad i=1,\cdots,n-1.
\end{equation*}

  Hence
  \begin{equation*}
  g \det \big(D^2 g + \frac1{\sigma_p}g^{-1} Dg\otimes Dg\big)\approx \det (g_{{\xi}^{(k)}{\xi}^{(l)}})_{k,l=1}^{n-1}\approx 1
  \end{equation*}
  and \begin{equation}\label{vnn}
\begin{split}
v_{\xi^{(n)}\xi^{(n)}} 
  =o\left(g^{\frac{1}{\sigma_p} -1}\right)+\frac{1}{\sigma_p}|Dg|^2 g^{\frac 1{\sigma_p}-1}\approx g^{\frac 1{\sigma_p}-1}\approx r^{1-\sigma_p}.
\end{split}
\end{equation}
By equation \eqref{g-eq}, we have $g_t\approx 1$. It implies  $|u_t|=|v_t|\approx r$.
Let $\lambda_1\le \cdots\le \lambda_n$ be the eigenvalues of $D^2 v$.
Then by \eqref{vxx} and \eqref{vnn}, we see that
$\lambda_1, \cdots, \lambda_{n-1} \approx r$ and  $\lambda_n\approx r^{1-\sigma_p} $.

  Let $\nu$ be the direction corresponding to the maximal eigenvalue $\lambda_n$.
To calculate the angle $\theta$ between $\nu$ and $\xi^{(n)}$,
let $\xi' \in \text{span} ({\xi}^{(1)},\cdots,{\xi}^{(n-1)} )$ be the unit vector such that
$\nu, \xi^{(n)}$ and $\xi'$ lie in a 2-dim plane. Then
$\nu= \xi^{(n)} \cos\theta+\xi' \sin\theta$.
Hence we have
$$r \gtrsim v_{\xi'\xi'}\ge v_{\nu\nu} \sin^2\theta \gtrsim r^{1-\sigma_p} \sin^2\theta.$$
Hence we obtain
$\sin^2\theta  \lesssim r^{\sigma_p} $.

By the duality between $u$ and $v$,
 $\lambda_1^{-1},  \cdots, \lambda_n^{-1}$ are the eigenvalues of $D^2 u$ at $x_0$.
In the above we have shown that
\beq\label{d2u3}
 \lambda_1^{-1}\approx \cdots\approx \lambda_{n-1}^{-1} \approx r^{-1}.
 \eeq
Hence
$${\begin{split}
  u_{rr}  &\le  \lambda_n^{-1}\cos^2 \theta+\lambda_1^{-1} \sin^2\theta\\
             & \lesssim  r^{\sigma_p -1}\cos^2 \theta+ r^{-1} \sin^2\theta
              \lesssim  r^{\sigma_p -1}
             \end{split}} $$
 and \eqref{3-urr} follows.
\end{proof}

\begin{lemma}\label{lemfbs-1}
For any given $T\in (0,T^*)$, we have the estimate
\begin{equation}\label{vt}
v_t \lesssim |Dv|  \ \ \forall\    t\in[0,T].
\end{equation}
\end{lemma}

\begin{proof}
The following proof is inspired by \cite{DL2004}.
By Proposition  \ref{DH1999} and Lemma \ref{urr-g-3},  we see that  \eqref{vt} holds for $t\in [0, T_0]$.
It suffices to verify \eqref{vt} for $t\in (T_0, T]$.
Denote $$v_\eps (y,t)=\frac{v\big((1+\eps )y, (1-A\eps )t\big)}{1+B\eps } , $$
where $A,B$ are two positive constants to be determined.
By direct computation,
\begin{equation}\label{101}
v_{\eps ,t}=  \eta\,
\frac{ (\det D^2v_\eps )^p}{(1+|Dv_\eps |^2)^{\frac{(n+2)p-1}{2}}} ,
\end{equation}
where
$$\eta=: \frac{(1-A\eps )(1+B\eps )^{np-1} (1+|Dv_\eps |^2)^{\frac{(n+2)p-1}{2}}}
{(1+\eps )^{2np} \big(1+\big(\frac{1+B\eps }{1+\eps }\big)^2|Dv_\eps |^2 \big)^{\frac{(n+2)p-1}{2}}} .$$
By Taylor's expansion,
\begin{equation*}
\begin{split}
1+\left(\frac{1+B\eps }{1+\eps }\right)^2|Dv_\eps |^2
 & =1+\big(1+2(B-1)\eps +O(\eps ^2)\big) |Dv_\eps |^2\\
&=(1+|Dv_\eps |^2) \left(1+2(B-1)
{\small\text{$\frac{|Dv_\eps |^2}{1+ |Dv_\eps |^2}$}} \eps  +O(\eps ^2)\right).
\end{split}
\end{equation*}
Hence
{\small
\begin{equation*}
\begin{split}
 \eta
&=\left(1+ (-A+(np-1)B-2np )\eps +O(\eps ^2)\big)
   \big(1-((n+2)p-1)(B-1) {\small\text{$\frac{ |Dv_\eps |^2}{1+|Dv_\eps |^2}$}} \eps  +O(\eps ^2)\right)\\
&=1+\Big(-A+(np-1)B-2np-((n+2)p-1)(B-1) {\small\text{$\frac{|Dv_\eps |^2}{1+|Dv_\eps |^2}$}} \Big) \eps +O(\eps ^2)\\
& \ge 1+\eps  +O(\eps ^2)
\end{split}
\end{equation*}
}
if $A\in (0,1)$, $B=\frac{3np+3}{n-1}$, and $|D v_\eps |\le \frac{np-1}{4(n+2)p}$.
The latter is true in
$$\Sigma(\delta_0)=\{(y,t)\ |\ v(y,t)<\delta_0, 0 < t\le T\}\ \ \ \text{for $\delta_0>0$ small},$$
by  $v\in C^{1,\alpha}$ and $Dv=0$ when $v=0$.
Hence
\begin{equation}
v_{\eps ,t}\ge \frac{(\det D^2v_\eps )^p}{(1+|Dv_\eps |^2)^{\frac{(n+2)p-1}{2}}}\ \ \text{in}\  \Sigma(\delta_0)
\end{equation}
when $\eps >0$ is small.

Next we want to apply the comparison principle to $v$ and $v_\eps$ in $\Sigma(\delta_0)$.
To compare the values of $v$ and $v_\eps $ on the parabolic boundary $\p_p\Sigma(\delta_0)$,
we compute
\begin{equation*}
\frac{d\tilde v_\eps(y, 0)}{d\eps } |_{\eps =0}
 =-\frac{\sigma_p}{1+\sigma_p}B\tilde v(y, 0) +y\cdot D\tilde v(y, 0)  ,
\end{equation*}
where for brevity we denote $\tilde v_\eps= v_\eps ^{\frac{\sigma_p}{1+\sigma_p}}$ and
$\tilde v= v^{\frac{\sigma_p}{1+\sigma_p}}$.
By the non-degeneracy condition (I2), we have
$${\begin{split}
 & 0\le \tilde v_\eps(y, 0)\lesssim \text{dist}(y, \Gamma_0),\\
 & |D\tilde v(y, 0)|\ge \lambda_0.
 \end{split}} $$
By \eqref{rho-0} and the uniform convexity of $\Gamma_0$, it implies that
$y\cdot D[\tilde v(y, 0)] \ge C \rho_0\lambda_0 $.
Hence we obtain
$$
\frac{d\tilde v_\eps (y, 0)}{d\eps } |_{\eps =0}\ge \frac C2 \rho_0\lambda_0>0 ,
$$
when $\eps$ and $\delta_0$ are small.
It implies that
\beq\label{vy0}
v_\eps (y, t)\ge v(y, t) \ \   \text{on} \ \   \partial_p \Sigma(\delta_0)\cap\{t=0\}.
\eeq
On the remaining part of the parabolic boundary $\partial_p \Sigma(\delta_0)\cap\{t>0\}$, we compute
\begin{equation}\label{dve}
\begin{split} \frac{dv_\eps (y,t)}{d\eps }|_{\eps =0 }
&=- Bv(y,t)+y\cdot Dv(y,t)- A t\, v_t(y,t).
\end{split}
\end{equation}
We claim that
$$- Bv(y,t)+y\cdot Dv(y,t)\ge v(y, t) =\delta_0$$
when $\delta_0$ is sufficiently small.
To see this, consider the one dimensional convex function
$\varphi(s)=v(sy,t)$. Choose $s_0\in (0, 1)$ such that $s_0 y\in \Gamma_t$.
Then by $\varphi(s_0)=0$ and the convexity, $\varphi(1)\le (1-s_0) \varphi'(1)$,
i.e., $v(y, t)\le (1-s_0) y\cdot Dv(y, t)$.
Hence
$$y\cdot Dv(y,t)\ge \frac {1}{1-s_0} v(y,t)\ge \frac{\delta_0}{1-s_0}.$$
Note that $s_0\to 1$ when $\delta_0\to 0$.
The claim follows.

By \cite{C1985}, $ \|v_t\|_{L^\infty(\partial_p \Sigma(\delta_0)\cap \{t>0\})}$ is uniformly bounded.
Hence by \eqref{dve},
\beq\label{vyt}
 \frac{dv_\eps (y,t)}{d\eps }|_{\eps =0 }
\ge \delta_0 - A t \, \|v_t\|_{L^\infty(\partial_p \Sigma(\delta_0)\cap \{t>0\})}\ge 0
\eeq
when $A>0$ is small.

Combining \eqref{vy0} and \eqref{vyt} yields
$v_\eps (y,t)\ge v(y,t)\ \ \text{on}\  \p_p\Sigma(\delta_0)$.
By the comparison principle, we then obtain
\begin{equation*}
v_\eps (y,t)\ge v(y,t) \ \   \text{in} \ \   \Sigma(\delta_0).
\end{equation*}
Differentiating the above inequality at $\eps =0$, by \eqref{dve} we obtain
\begin{equation}\label{bav}
\begin{split}
0\ge Bv(y,t)-y\cdot Dv(y,t)+At\, v_t.
\end{split}
\end{equation}
As noted at the beginning, it suffices to consider the case $t>T_0$.
When $t>T_0$, from \eqref{bav}, we obtain
\begin{equation}
v_t\le \frac{diam(\mathcal M_0) |Dv|}{AT_0}.
\end{equation}
Lemma \ref{lemfbs-1} is proved.
\end{proof}

\begin{corollary}\label{cor-2.1}
Let $(\theta,r_\eps  (\theta,t))$, $\theta\in\mathbb S^{n-1}$
be the spherical parametrization of $\{v(y,t)=\eps \}$ for $\eps >0$ small.
Then we have
\beq\label{dre}
-\frac{dr_\eps  (\theta,t)}{dt} \lesssim 1  \ \ \forall\    t\in[0,T].
\eeq
\end{corollary}

\begin{proof}
Differentiating $v(\theta,r_\eps  (\theta,t),t)=\eps$ in $t$ yields
\begin{equation*}
\frac{dr_\eps  (\theta,t)}{dt}\cdot\left(\nabla v\cdot \frac{y}{|y|}\right)+v_t=0.
\end{equation*}
Hence \eqref{dre} follows from Lemma \ref{lemfbs-1}.
\end{proof}

Let $u(\cdot, t)$ be the Legendre transform of $v(\cdot, t)$. Then we have the following corollary.

\begin{corollary}\label{ut-sup}
We have the estimate
\beq \label{utx0}
-u_t(x,t) \lesssim |x|  \ \ \forall\    t\in[0,T].
\eeq
\end{corollary}
\begin{proof}
This follows from the duality between $u$ and $v$.
\end{proof}

Lemma \ref{lemfbs-1} and Corollaries \ref{cor-2.1} - \ref{ut-sup}
imply that the interface $\Gamma_t$ moves at finite speed.
Next we show the interface $\Gamma_t$ moves at positive speed.
\begin{lemma}\label{ut-sub}
We have
\beq \label{utx}
-u_t(x,t) \gtrsim  |x|   \ \ \forall\    t\in[0,T].
\eeq
\end{lemma}

\begin{proof}
Let $\widehat G=\frac{-u_t+\varepsilon}{u}$, where $\varepsilon>0$ is a constant.
Suppose the infimum $\inf_{S_T} \widehat G(x,t)$ is attained at the point $(x_0, t_0)$,
where $S_T= \{(x, t) \ | \ u(x,t)<1, 0<t\le T\}$.
Since  $\varepsilon>0$, we see that $x_0\ne 0$.
If $(x_0, t_0)\not\in \p_p S_T$, the parabolic boundary of $S_T$,
then at $(x_0, t_0)$ we have
\begin{equation}\label{2-2-1}
\begin{split}
&0=(\log  \widehat G)_i=\frac{u_{ti}}{u_t-\varepsilon}-\frac{u_i}{u},\\
&0\le (\log  \widehat G)_{ii}=\frac{u_{tii}}{u_t-\varepsilon}-\frac{u_{ti}^2}{(u_t-\varepsilon)^2}- \Big(\frac{u_{ii}}{u}-\frac{u_i^2}{u^2}\Big),\\
&0\ge (\log  \widehat G)_t=\frac{u_{tt}}{u_t-\varepsilon}-\frac{u_t}{u}.
\end{split}
\end{equation}
By a rotation of coordinates, we may assume $D^2u(x_0,t_0)$ is diagonalized. Hence,
\begin{equation}
\begin{split}
0&\le \frac{(\log  \widehat G)_t}{u_t}+pu^{ii}(\log  \widehat G)_{ii}\\
&=\frac{1}{u_t-\varepsilon}\Big(\frac{u_{tt}}{u_t}+pu^{ii}u_{tii}\Big)-\frac{(np+1)}{u}-p\frac{u^{ii}u_{ti}^2}{(u_t-\varepsilon)^2}+\frac{p u^{ii}u_i^2}{u^2}\\
&= -\frac{(pn+1)}{u}<0.
\end{split}
\end{equation}
This contradiction implies that $(x_0, t_0)$ must be a point
on the parabolic boundary of $S_T$.
Sending $\varepsilon\to 0$,  by Lemma \ref{urr-g-3} and Proposition \ref{DS2009},
we obtain \eqref{utx}.
\end{proof}

Similarly to Corollaries \ref{cor-2.1} - \ref{ut-sup}, we have

\begin{corollary}\label{cor-2.3}
Let $(\theta,r_\eps  (\theta,t))$, $\theta\in\mathbb S^{n-1}$ be the spherical parametrization of $\{v(y,t)=\eps \}$ for $\eps >0$ small.
We have
\beq\label{dre1}
-\frac{dr_\eps  (\theta,t)}{dt} \gtrsim 1  \ \ \forall\     t\in[0,T].
\eeq
\end{corollary}

\begin{corollary}\label{vt-inf}
We have the estimate
\beq\label{vt1}
v_t(x,t) \gtrsim  |\nabla v|  \ \ \forall\     t\in[0,T].
\eeq
\end{corollary}

\vspace{5mm}
\section{Growth estimates at the singular point}\label{s3}
\vspace{3mm}

Recall that $u(\cdot, t)$ satisfies equation \eqref{u}, namely
\begin{equation}\label{3.1}
\det D^2 u =\frac{1}{(-u_t)^{\frac1p}(1+|x|^2)^{\frac{(n+2)p-1}{2p}}}+ c_t\delta_0.
\end{equation}
In Section \ref{s2}, we proved the growth estimates \eqref{utx0} and \eqref{utx} for $u_t$ near the origin.
In this section, we establish crucial growth estimates for $w=u-\phi$ at the origin,
where $\phi(\cdot, t)$ is the tangential cone of $u(\cdot, t)$ at $(0, t)$, namely,
$\phi(\cdot, t)$ is a homogeneous function of degree one satisfying
$|u(x,t)-\phi(x,t)|=o(r)$ as $r=|x|\to 0$, for any given $t\in [0, T]$.

\begin{lemma}\label{lemc1}
Near the origin, we have
\begin{equation}\label{l3.1}
|u(x,t)-\phi(x,t)|\le r\omega(r)  \ \ \forall\   t\in [0,T]
\end{equation}
for a function $\omega(r)\to 0$ as $r\to 0$ independent of $t\in [0, T]$, where $r=|x|$.
\end{lemma}

\begin{proof}
If \eqref{l3.1} is not true,
there exists a sequence $(x_k,t_k)\rightarrow (0,\bar t)$ such that
\begin{equation}\label{u-pex}
u(x_k,t_k)-\phi(x_k,t_k)\ge \eps _0 |x_k|
\end{equation}
for some constant $\eps _0>0$.  Make the scaling
$u_k(x)=\frac{u(r_k x,t_k)}{r_k}$,
where $r_k=|x_k|$.
By Corollary \ref{ut-sup} and Lemma \ref{ut-sub} we have
\begin{equation*}
\Big|u_k(x)-\frac{u(r_k x,\bar t)}{r_k}\Big| \le C|x||t_k-\bar t|\rightarrow 0  \ \   \text{as}  \ \   k\rightarrow +\infty,
\end{equation*}
which implies that $u_k(x) \rightarrow \phi(x,\bar t)$ locally uniformly in $\mathbb R^n$.
Notice that the interface $\Gamma_t$ moves at finite speed.
Hence $\phi(x,t_k)\rightarrow \phi(x,\bar t)$ locally uniformly in $\mathbb R^n$.
Hence we obtain
\begin{equation*}
u_k\Big(\frac{x_k}{r_k}\Big)-\phi\Big(\frac{x_k}{r_k},t_k\Big) \rightarrow 0  \ \   \text{as}  \ \   k\rightarrow +\infty,
\end{equation*}
in contradiction with \eqref{u-pex}.
\end{proof}

Let $f(x, t)$ be a function defined in ${ Q }\subset \R^n\times \R$.
We say $f$ is parabolically convex if it is convex in $x$ and non-increasing in $t$.
Denote ${ Q }(t)=\{x \ | \  (x,t)\in { Q }\}$, $ \underline t =\inf \{t\ |  \ { Q }(t)\neq \emptyset\}$.
The parabolic boundary of ${ Q }$ is
\begin{equation*}
\partial_p { Q } =\cup_{t} \big\{\partial { Q }(t) \times \{ t\} \big\}\cup \big\{{{ Q }( \underline t )} \times \{ \underline t\} \big\} .
\end{equation*}
We say  ${ Q }$ is bowl-shaped
if ${ Q }(t)$ is convex for each $t$ and ${ Q }(t_1)\subset { Q }(t_2)$ for $t_1\le t_2$.

\begin{lemma}\label{CH-IUMJ}
Let ${ Q }\subset \mathbb{R}^{n}\times\R$ be a bounded bowl-shaped domain.
Let $u\in C^4({ Q })\cap C^0(\overline { Q })$ be a parabolically convex function satisfying the equation
\begin{equation}\label{pogo-eq}
\begin{split}
-u_t (\det D^2u)^p &= f(x) \ \   \text{in} \ \   \overline { Q } \backslash \partial_p{ Q },\\
 u&=0 \ \   \text{on} \ \   \partial_p { Q },
\end{split}
\end{equation}
where $p>0$. Then we have the estimate
\begin{equation}\label{pogo-es}
(-u)|D^2 u|\le C
\end{equation}
for a constant $C>0$ depending only on $n, p, \|\p_xu\|_{L^\infty({ Q })}, \|\log f\|_{C^{1,1}({ Q })}$.
\end{lemma}

Pogorelov type estimates can be found in many articles.
For parabolic Monge-Amp\`ere equations it can be found in \cite{GH1998} and \cite{XB}.
Estimate \eqref{pogo-es} can be found in \cite{XB}, by choosing $\rho(t)= - e^{-pt}$ there.

Applying estimate \eqref{pogo-es} to $u-\ell $ for a proper  linear function $\ell$,
we obtain the following corollary similarly as  \cite{HTW, S05}.
Note that we need Lemma \ref{lemc1} to guarantee $|u-\ell |\approx |x|$ uniformly.

\begin{corollary}\label{C3.1}
Let $u(\cdot, t)$ be the Legendre transform of $v(\cdot, t)$ which satisfies equation \eqref{3.1}.
Then we have
\beq\label{po-es1}
|x| |D^2 u(x,t)|\lesssim 1 \ \ \text{for} \  (x,t)\in B_1(0)\backslash \{0\}\times[0,T].
\eeq
\end{corollary}

By a rescaling argument, from \eqref{po-es1} we then have

\begin{corollary}\label{C3.2}
 There holds
\begin{equation}\label{p-sup}
|x| |D^2\phi(x,t)| \lesssim 1 \ \  \text{for}  \ \   (x,t)\in \mathbb R^n\backslash \{0\}\times[0,T].
\end{equation}
\end{corollary}

To establish the a priori estimates in Theorem \ref{thmA},
we assume that the interface  $\Gamma_t$ is smooth and uniformly convex for $t\in (0, T]$.
By Corollary \ref{ut-sup} and Lemma \ref{ut-sub}, we have
\begin{equation}\label{ut-b-3}
-u_{t}(x,t)\approx |x|  \ \  \text{for}  \ \   (x,t)\in B_1(0)\backslash \{0\}\times[0,T].
\end{equation}
By estimate \eqref{ut-b-3} and equation \eqref{3.1}, we have
\beq\label{d2u1}
\det D^2 u \approx |x|^{-1/p}\ \ \text{near the origin.}
\eeq
Near the origin, $u$ is asymptotic to the convex cone $\phi$, which suggests that
\beq\label{d2u2}
C_1|x|^{-1}\le u_{\xi\xi}(x, t)\le C_2|x|^{-1}
\eeq
for any unit vector $\xi\perp \vec{ox}$,
where the positive constants $C_1, C_2$ depend only on $\mathcal M_0, n, p, T$.
The second inequality in \eqref{d2u2} follows from \eqref{po-es1}.
The first inequality shall be proved below.

\begin{lemma}\label{urr-sup}
There holds the estimate
\begin{equation}\label{urr}
u_{rr}(x,t) \lesssim |x|^{n-1-1/p}  \ \  \text{for}  \ \   (x,t)\in B_1(0)\backslash \{0\}\times[0,T].
\end{equation}
\end{lemma}

\begin{proof}
By \eqref{ut-b-3},  we see that \eqref{urr} is equivalent to
\begin{equation}\label {urrt}
r^2 u_{rr}  \lesssim (-u_t)^{n+1-1/p}.
\end{equation}

Introduce the auxiliary function
$$G(x,t)=:\frac{x_ix_j u_{ij}}{(-u_t)^{\beta-4(p\beta+1)u_t}}
     \ \   \text{in} \ \   \tilde\Sigma(\delta_0)=:B_{\delta_0}(0)\times (0,T] ,$$
where the constant $\beta \in (1, n+1- \frac 1p)$, and $\delta_0>0$ is a small positive constant.
Assume that the maximum
$\max_{\tilde\Sigma(\delta_0)} G(x,t)$ is attained at  $(\bar x,\bar t)$.
Since $\beta<n+1- \frac 1p$, by Lemma \ref{urr-g-3}, we see that $\bar x\ne 0$.
By Proposition \ref{DS2009},
$G$ is under control on the parabolic boundary $\p_p\tilde\Sigma(\delta_0)$.

Therefore we may assume that  $(\bar x,\bar t)$ is an interior point of $\tilde\Sigma(\delta_0)$ and $\bar x\ne 0$.
One easily verifies that $r^2 u_{rr}$ is invariant under linear transformations of coordinates.
Indeed,  let $\tilde x=Ax$ and $\tilde u(\tilde x)=u(A^{-1}\tilde x)=u(x)$,
where $A=(a_{ij})_{i,j=1}^n$, $A^{-1}=(a^{ij})_{i,j=1}^n$.
Then
\begin{equation*}
\begin{split}
\tilde x_i\tilde x_j \tilde u_{ij}=a_{ik}x_ka_{jl}x_l u_{st}a^{s i}a^{tj}=x_kx_l u_{kl}.
\end{split}
\end{equation*}
Hence we may assume that $\bar x=(r,0,\cdots,0)$ with $0<r<\delta_0$.

We then make a linear transform of the coordinates, which leaves the origin and the point  $\bar x=(r,0,\cdots,0)$  unchanged,
such that the matrix $\{u_{ij}(\bar x,\bar t)\}$ is diagonal.
A direct calculation yields that, at $(\bar x,\bar t)$,
\begin{equation}\label{3-5-1}
\begin{split}
 \ \  \,\,\,0=(\log  G)_a=\frac{2u_{1a}+ru_{11a}}{ru_{11}}-\big(\beta-4(p\beta+1)u_t -4(p\beta+1)u_t\log  (-u_t)\big)\frac{u_{ta}}{u_t},
\end{split}
\end{equation}
\begin{equation}\label{3-5-2}
\begin{split}
0\ge (\log  G)_{aa}&=\frac{2u_{aa}+4ru_{1aa}+r^2 u_{11aa}}{r^2 u_{11}}-\frac{(2u_{1a}+ru_{11a})^2}{r^2 u_{11}^2}\\
& \quad -\big(\beta-4(p\beta+1)u_t-4(p\beta+1)u_t\log (-u_t)\big) \Big(\frac{u_{taa}}{u_t}-\frac{u_{ta}^2}{u_t^2}\Big)\\
& \quad + 4(p\beta+1)\big(2u_t+u_t\log  (-u_t)\big)\frac{u_{ta}^2}{u_t^2},
\end{split}
\end{equation}
and
\begin{equation}\label{3-5-3}
\,\,\, 0\le (\log  G)_t=\frac{u_{11t}}{u_{11}}-\big(\beta-4(p\beta+1)u_t- 4(p\beta+1)u_t\log  (-u_t)\big)\frac{u_{tt}}{u_t}. \ \   \ \   \ \
\end{equation}
Differentiating equation \eqref{3.1} gives
\begin{equation}\label{3-5-4}
\begin{split}
\frac{u_{tt}}{u_t}+pu^{aa}u_{taa}=0,
\end{split}
\end{equation}
\begin{equation}\label{3-5-5}
\begin{split}
\frac{u_{ti}}{u_t}+pu^{aa}u_{aai}=(\log  f)_i,
\end{split}
\end{equation}
and
\begin{equation}\label{3-5-44}
\begin{split}
\frac{u_{tii}}{u_t}+pu^{aa}u_{aaii}= \frac{u_{ti}^2}{u_t^2}+pu^{ak}u^{bl}u_{abi}u_{kli}+(\log  f)_{ii},
\end{split}
\end{equation}
where $\log  f= -\frac{(n+2)p-1}{2}\log (1+|x|^2)$.
Hence
\begin{equation}
\begin{split}\label{3-5-55}
0&\ge  \frac{(\log  G)_t}{u_t}+pu^{aa}(\log  G)_{aa}\\
&=\frac{1}{u_{11}}\Big(\frac{u_{11t}}{u_t}+pu^{aa}u_{11aa}\Big)-\frac{\hat \beta}{u_t}\Big(\frac{u_{tt}}{u_t}+pu^{aa}u_{taa}\Big)+\frac{2p(n-2)}{r^2 u_{11}}\\
& \quad -\frac{pu^{aa}u_{11a}^2}{u_{11}^2}+\sum_{a\ge 2}\frac{4pu^{aa}u_{1aa}}{ru_{11}}+ p\big(\beta +4(p\beta+1)u_t\big)u^{aa}\frac{u_{ta}^2}{u_t^2},
\end{split}
\end{equation}
where
\begin{equation*}
\hat \beta=\beta-4(p\beta+1)u_t-4(p\beta+1)u_t\log (-u_t).
\end{equation*}
Now we estimate \eqref{3-5-55} term by term.
By \eqref{3-5-44}, we have
\begin{equation}\label{3-5-6}
\begin{split}
& \quad \frac{1}{u_{11}}\Big(\frac{u_{11t}}{u_t}+pu^{aa}u_{11aa}\Big) -p\frac{u^{aa}u_{11a}^2}{u_{11}^2}\\
&=\frac{u_{t1}^2}{u_{11}u_t^2}+pu^{aa}u^{bb}u^{11}u_{ab1}^2+ u^{11}(\log  f)_{11}-p\frac{u^{aa}u_{11a}^2}{u_{11}^2}\\
&\ge\frac{u_{t1}^2}{u_{11}u_t^2}+pu^{11}\sum_{a\ge 2}(u^{aa})^2u_{aa1}^2+u^{11}(\log  f)_{11}\\
&\ge\frac{u_{t1}^2}{u_{11}u_t^2}+p\frac{(\sum_{a\ge 2}u^{aa}u_{aa1})^2}{(n-1)u_{11}}+u^{11}(\log  f)_{11}.
\end{split}
\end{equation}
Combining \eqref{3-5-1} and \eqref{3-5-5} yields
\begin{equation}\label{3-5-7}
(p\hat \beta+1)\frac{u_{t1}}{u_t}=\frac {2p}r -p\sum_{a\ge 2}u^{aa}u_{1aa}+(\log  f)_1.
\end{equation}
Denote $\mathcal K=\sum_{a\ge 2}u^{aa}u_{1aa}$.  Inserting \eqref{3-5-6}, \eqref{3-5-7} into \eqref{3-5-55}
and multiplying with $u_{11}$, we obtain
\begin{equation}\label{3-5-8}
\begin{split}
0&\ge \frac{1+p(\beta+4(p\beta+1)u_t)}{(p\hat \beta+1)^2}\Big(\frac {2p}r -p\mathcal K+(\log  f)_1\Big)^2\\
& \ \ \  +\frac{2p(n-2)}{r^2}+ \frac{4p}{r}\mathcal K+\frac{p\, \mathcal K^2}{n-1}+(\log  f)_{11}.
\end{split}
\end{equation}
By \eqref{ut-b-3}, $|u_t|\lesssim \delta_0$ is small, there holds
\begin{equation}\label{3-5-9}
\begin{split}
& \quad  \frac{1+p(\beta+4(p\beta+1)u_t)}{(p\hat \beta+1)^2}\\
&= \frac{1}{p\beta+1}\frac{1+4pu_t}{(1-4pu_t-4pu_t\log (-u_t))^2}\\
&=\frac{1}{p\beta+1}\Big(1+12pu_t+8pu_t\log (-u_t)+ o\big(u_t\log (-u_t)\big)\Big)\\
&\ge\frac{1+2pu_t\log (-u_t)}{p\beta+1}.
\end{split}
\end{equation}
Hence,
\begin{equation}\label{3-5-10}
\begin{split}
& \quad  \frac{1+p(\beta+4(p\beta+1)u_t)}{(p\hat \beta+1)^2}\Big(\frac {2p}r -p\mathcal K+(\log  f)_1\Big)^2\\
&\ge\frac{1+2pu_t\log (-u_t)}{p\beta+1}\Big[p^2\Big(\frac 2r -\mathcal K\Big)^2
                       +2p\Big(\frac 2r -\mathcal K\Big)(\log  f)_1+\big((\log  f)_1\big)^2\Big]\\
&\ge \frac{p^2(1+pu_t\log (-u_t))}{p\beta+1}\Big(\frac 2r -\mathcal K\Big)^2
                      + \frac{1+2pu_t\log (-u_t)}{p\beta+1} \Big( 1 - \frac{1+2pu_t\log (-u_t)}{pu_t\log (-u_t)}\Big)\big((\log  f)_1\big)^2\\
&= \frac{p^2(1+pu_t\log (-u_t))}{p\beta+1}\Big(\frac 2r -\mathcal K\Big)^2
                     - \frac{(1+2pu_t\log (-u_t)) (1+pu_t\log (-u_t)) }{p(p\beta+1)u_t\log (-u_t)} \big((\log  f)_1\big)^2 \\
&\ge  \frac{p^2(1+pu_t\log (-u_t))}{p\beta+1}\Big(\frac 2r -\mathcal K\Big)^2- \frac{ C_{n,p}}{u_t\log (-u_t)}
\end{split}
\end{equation}
where $C_{n,p}$ is a constant depending only on $n$ and $p$.
Inserting \eqref{3-5-10} into \eqref{3-5-8} yields that
\begin{equation}\label{3-5-11}
\begin{split}
0&\ge\frac{4p(p(\beta-1)+1-p^2u_t\log (-u_t))}{p\beta+1}\frac{\mathcal K}{r}
                 +\Big(\frac{p^2(1+pu_t\log (-u_t))}{p\beta+1}+\frac{p}{n-1}\Big) \mathcal K^2\\
& \quad  + \Big(\frac{4p^2(1+pu_t\log (-u_t))}{p\beta+1}+2p(n-2)\Big) \frac 1{r^2} -\frac{ C_{n,p}}{u_t\log (-u_t)}+(\log  f)_{11} \\
&\ge \Big(\frac{p^2}{p\beta+1}+\frac{p}{n-1}\Big) \mathcal K^2
               -\frac{4p[p(\beta-1)+1]}{p\beta+1}\frac{|\mathcal K|}{r} + \Big(\frac{4p^2}{p\beta+1}+2p(n-2)\Big)\frac{1}{r^2} \\
& \quad  +\frac{4p^3u_t\log (-u_t)}{p\beta+1}\frac{1}{r^2}- \frac{ C_{n,p}}{u_t\log (-u_t)} - { C_{n,p} },
\end{split}
\end{equation}
as $\big|p(\beta-1)+1-p^2u_t\log (-u_t)\big| < p(\beta-1)+1$ for $|u_t|$ small.
Notice that
\begin{equation*}
\begin{split}
&\frac{4p^2(p\beta+1-p)^2}{(p\beta+1)^2}-p^2\Big(\frac{p}{p\beta+1}+\frac{1}{n-1}\Big) \Big(\frac{4p}{p\beta+1}+2(n-2)\Big)\\
=&\frac{2np^3[\beta-(n+1-1/p)]}{(n-1)(1+p\beta)}<0.
\end{split}
\end{equation*}
Since $|u_t|$ is small for $\delta_0$ small, therefore, \eqref{3-5-11} reduces to
\begin{equation*}
0\ge \frac{4u_t\log (-u_t)}{p\beta+1}\frac{1}{r^2} - \frac{ C_{n,p}}{u_t\log (-u_t)} - { C_{n,p} }>0
\end{equation*}
which is impossible.

The above argument implies that the auxiliary function $G$ cannot attain its maximum at an interior point in
$\tilde\Sigma(\delta_0)$.
Hence $\max G$ must be attained on the parabolic boundary of $\tilde\Sigma(\delta_0)$.
Sending $\beta\to n-1-\frac 1p$, we obtain estimate \eqref{urr}.
\end{proof}

\begin{corollary}\label{C3.3}
Let $\lambda_1(x,t)\le \cdots \le\lambda_n(x,t)$ be the eigenvalues of $D^2 u$
at the point $(x,t)\in B_1(0)\backslash \{0\}\times[0,T]$.
Then,
\begin{equation} \label{eiges}
{\begin{split}
 &\lambda_1(x,t)\approx |x|^{n-1-1/p}, \\
 &\lambda_2(x,t)\approx \cdots \approx\lambda_{n}(x,t) \approx |x|^{-1}.
 \end{split}}
\end{equation}
\end{corollary}

\begin{proof} \eqref{eiges} follows from
\eqref{urr}, \eqref{d2u1} and   \eqref{d2u3}.
\end{proof}

Note that the inequality \eqref{d2u2} follows from \eqref{eiges}.
From \eqref{d2u1} we also have $u_{rr}(x,t) \gtrsim |x|^{n-1-1/p}$. Hence
\eqref{urr} can be strengthened to
\begin{equation}\label{urr2}
u_{rr}(x,t) \approx |x|^{n-1-1/p}  \ \  \text{for}  \ \   (x,t)\in B_1(0)\backslash \{0\}\times[0,T].
\end{equation}
Therefore by taking integration,
\begin{equation}\label{asyw}
w(x,t)\approx |x|^{n+1-1/p} \ \   \forall\ (x,t)\in B_1(0) \backslash \{0\}\times[0,T].
\end{equation}
By \eqref{d2u2} and a rescaling argument, we also have

\begin{corollary}\label{C3.4}
Let $\phi(\cdot, t)$ be the asymptotic cone of $u(\cdot, t)$ at $(0, t)$.
Then
\begin{equation}\label{p-sub}
|x| |D^2_\xi\phi(x,t)| \approx 1  \ \   \forall\ (x,t)\in \mathbb R^n \backslash \{0\}\times[0,T],
\end{equation}
for any unit vector $\xi\perp \vec{ox}$.
\end{corollary}

\begin{corollary}\label{C3.5}
For any given $T\in (0, T^*)$, there hold
\beq \label{nondc-a}
 \lambda_{\eps, t,  i}\approx 1,\quad \forall~ t\in [0,T]
 \eeq
where $\lambda_{\eps, t,  i} \ (i=1, \cdots, n-1)$ are the principal curvatures of the level set $\{y|v(y, t)=\eps\}$,
for  $\eps\ge 0$ small.  
\end{corollary}

\begin{proof}
By \eqref{urr2}, we have
\begin{equation*}
\begin{split}
v(y,t)=x\cdot D u-u=ru_r-u=\int_0^r \int_{\lambda}^r u_{\rho\rho}d\rho d\lambda\approx  r^{1+\sigma_p} .
\end{split}
\end{equation*}
 Hence
\begin{equation*}
|Dg(y,t)|\approx \frac{|D_y v|}{v^{1/(1+\sigma_p)}}\approx 1
\end{equation*}
and $g\approx r^{\sigma_p}$. 
We can then adapt the proof of Lemma \ref{urr-g-3} to show that $\lambda_{\eps, t,  i}\approx 1$.
\end{proof}

We now express equation \eqref{3.1} in the spherical coordinates $(\theta,r)$,
where $r=|x|$ and $\theta=(\theta_1, \cdots, \theta_{n-1})$ is  an orthonormal frame on $\mathbb S^{n-1}$.

\begin{lemma}\label{L2.4}
In the spherical coordinate $(\theta,r)$, we have
\begin{equation} \label{upb}
\begin{split}
    |\partial_r^k w(\theta,r,t)| & \lesssim r^{n-k+1-1/p}, \ \   k=0,1,2,\\
    |\partial_{\theta}^2 w(\theta,r,t) |  &\lesssim r,\\
   |\partial_{r\theta}w(\theta,r,t)|  & \lesssim r^{\frac {n-1/p}2},
\end{split}
\end{equation}
for any point $(\theta,r,t)\in \mathbb S^{n-1}\times(0,1]\times[0,T]$.
\end{lemma}

\begin{proof}
From Proposition \ref{DH1999} and Lemma \ref{urr-g-3}, it is enough to prove it for $t\in [T_0,T]$,
for a small $T_0>0$.

Given a time  $t_0\in[T_0,T]$,
let us suppose that $\phi(\lambda e_1,t_0)=\lambda$ and $\phi(x,t_0)\geq x_1$.
For any given $\eps >0$ small, denote
$${ Q }=\{(x,t)\in  B_1(0)\times(0,t_0] \ |\ u(x,t)<(1+\vep)x_1\}.$$
Let ${ Q }(t)=\{x \ | \ (x,t)\in { Q }\}$.
By the local strict convexity of $u$, we have
${ Q }(t_0) \subset\subset B_1(0)$ for $\eps $ small.
Set
$$\underline t=:\inf\{t \ |\  { Q }(t) \ne \emptyset\}.$$
According to Lemma \ref{ut-sub} and Corollary \ref{ut-sup},   one knows
\begin{equation}
t_0-\underline t\approx \eps .
\end{equation}
Since $u$ is smooth and parabolically convex away from the origin, ${ Q }$ is bowl-shaped and $u(x,t)-(1+\vep) x_1=0$ on the parabolic boundary $\partial_p { Q }$.

For any point $x=(x_1, \tilde x)\in { Q }(t_0)$, where $x_1>0$ and $\tilde x=(x_2, \cdots, x_n)$, we have
\begin{equation}\label{uphi}
\begin{split}
\vep x_1  &> u(x,t_0)-\phi(x_1,0,t_0) \\
&\ge \phi(x,t_0)-\phi(x_1,0,t_0)\\
&= D_{\tilde x}\phi(x_1,0,t_0)\cdot \tilde x +  {\tilde x}^T D^2_{\tilde x\tilde x}\phi(x_1,\sigma\tilde x,t_0) \tilde x,\quad \sigma \in (0,1)\\
 &\ge  c\frac{|\tilde x|^2}{x_1+|\tilde x|}.
\end{split}
\end{equation}
Hence, by Cauchy's inequality,
\begin{equation}\label{e-te}
{ Q }(t_0)\subset\big\{x\in B_1(0)\ | \ |\tilde x|\le \frac{\sqrt{1+2c}}{c}\,\vep^{\frac 12} x_1\big\}.
\end{equation}

Denote
\begin{equation}\label{te}
\begin{split}
\alpha_{t_0,\vep} & =\sup\{\alpha\ | \ \alpha e_1\in { Q }(t_0)\}, \ \   \\
\beta_{t_0,\vep} & =\sup\{x_1\ | \ x\in { Q }(t_0)\}.
\end{split}
\end{equation}
Then,
$$
u(\alpha_{t_0,\vep}e_1,t_0)-\phi(\alpha_{t_0,\vep}e_1,t_0) =\vep\alpha_{t_0,\vep}.
$$
By \eqref{asyw},  we have
$$
u(\alpha_{t_0,\vep}e_1,t_0)-\phi(\alpha_{t_0,\vep}e_1,t_0) \approx \alpha_{t_0,\vep}^{n+1-1/p} .
$$
This implies
$$ \beta_{t_0,\vep} \ge \alpha_{t_0,\vep}\approx \vep^{\frac p{pn-1}}.$$
By the definition of $\beta_{t_0,\vep}$ in \eqref{te}, and by the strict convexity of $u$,
there exists a unique $\tilde x_{t_0,\vep}$ such that
$(\beta_{t_0,\vep}, \tilde x_{t_0,\vep})\in \partial { Q }(t_0)$.
Hence by \eqref{asyw},
\begin{equation*}
\begin{split}
\vep \beta_{t_0,\vep} &=u(\beta_{t_0,\vep}, \tilde x_{t_0,\vep},t)-\phi(\beta_{t_0,\vep},0,t_0) \\
 &\ge \phi(\beta_{t_0,\vep}, \tilde x_{t_0,\vep},t_0) -\phi(\beta_{t_0,\vep},0,t_0)
      +C(\beta_{t_0,\vep}^2 +|\tilde x_{t_0,\vep}|^2)^{\frac {n+1-1/p}{2}}\\
 & \ge C\beta_{t_0,\vep}^{n+1-1/p},
\end{split}
\end{equation*}
which implies $\beta_{t_0,\vep}\le C\vep^{\frac p{np-1}}$.
Hence
$$\alpha_{t_0,\vep}\approx \beta_{t_0,\vep}\approx \vep^{\frac p{np-1}}.$$

Make the coordinate change $x\rightarrow y=T_{t_0}(x)$, given by
\begin{equation}
y_1=\frac{x_1}{\alpha_{t_0,\vep}}, \ \   y_k=\frac{x_k}{\vep^{\frac 12}\alpha_{t_0,\vep}}\ \ \ (k=2,\cdots,n).
\end{equation}
We claim that $T_{t_0}({ Q }(t_0))$ has a good shape, namely,
\begin{equation*}
T_{t_0}({ Q }(t_0)) \sim \{y=(y_1,\tilde y)\in\R^n\ | \ |\tilde y|< y_1, 0<y_1<1\}.
\end{equation*}
Indeed,
denote $\tilde { Q }(t_0) = { Q }(t_0)\cap\{x_1=\tau_{t_0} \alpha_{t_0,\vep}\}$, where $\tau_{t_0}>0$ is a small constant.
Then it suffices to prove
$$\tilde { Q }(t_0) \sim \{|\tilde x|<\vep^{1/2}\alpha_{t_0,\vep}\} , $$
as convex domains in $\R^{n-1}$.
From \eqref{e-te}, we have
$$
\tilde { Q }(t_0)\subset \{|\tilde x|<C\vep^{1/2}\alpha_{t_0,\vep}\}.
$$
For each point  $(x_1, \tilde x)\in\partial \tilde { Q }(t_0)$, $u(x_1, \tilde x,t_0)=(1+\vep) x_1$, and
by Corollary \ref{C3.2} and  \eqref{asyw}, we have
\begin{equation*}
\begin{split}
u(x_1, \tilde x,t_0)
   &\le \phi(x_1, \tilde x,t_0)+C\big(|x_1|^2+|\tilde x|^2\big)^{\frac{n+1-1/p}{2}} \\
& = \phi(x_1, 0,t_0) + D_{\tilde x}\phi(x_1,0,t_0)\cdot \tilde x + \tilde x^T D^2_{\tilde x\tilde x}\phi(x_1,\sigma \tilde x,t_0)\tilde x +C\big(|x_1|^2+|\tilde x|^2\big)^{\frac{n+1-1/p}{2}} \\
&  \le x_1 + C{\text{$\frac{|\tilde x|^2}{x_1}$}} + 2 C x_1^{n+1-1/p} + 2 C |\tilde x|^{n+1-1/p},
\end{split}
\end{equation*}
which yields 
$${\begin{split}
C(1+2 x_1|\tilde x|^{n-1-1/p})|\tilde x|^2
 &\ge \vep x_1^2 - 2 C x_1^{n+2-1/p} \\
 & = \tau_{t_0}^2 \vep \alpha_{t_0,\vep}^2 - 2C \tau_{t_0}^{n+2-1/p}\alpha_{t_0,\vep}^{n+2-1/p}.
\end{split} }$$
Choosing a small positive constant $\tau_{t_0} <<1$  such that
$$ \tau_{t_0}^{n+2-1/p}\alpha_{t_0,\vep}^{n+2-1/p} \approx \tau_{t_0}^{n+2-1/p}\vep \alpha_{t_0,\vep}^2 <<\tau_{t_0}^2 \vep \alpha_{t_0,\vep}^2,$$
we get $|\tilde x|\ge c\vep^{1/2} \alpha_{t_0,\vep}$ for some positive constant $c$. Thus,
$$ \{|\tilde x|<c\vep^{1/2} \alpha_{t_0,\vep}\}\subset \tilde { Q }(t_0).$$
The claim follows.

By \eqref{asyw}, we also have
\begin{equation}
\inf_{{ Q }(t_0)} \big(u(x,t_0)- (1+\vep)x_1\big)
 =  \inf_{{ Q }(t_0)} \big(\phi(x,t_0)+w(x,t_0) - (1+\vep) x_1\big)
\approx -\vep \alpha_{t_0,\vep}.
\end{equation}
Indeed, taking $x= \tau\alpha_{t_0,\vep}e_1$, we have \begin{equation*}
\begin{split}
\phi(x,t_0)+w(x,t_0) - (1+\vep) x_1 &= w(\tau\alpha_{t_0,\vep}e_1,t_0) - \vep\tau\alpha_{t_0,\vep} \\
&\le C(\tau\alpha_{t_0,\vep})^{n+1-1/p} - \vep\tau\alpha_{t_0,\vep} \\
& \le -\frac12 \vep\tau\alpha_{t_0,\vep}
\end{split}
\end{equation*}
for some small positive constant $\tau$.

Let
\begin{equation*}
\begin{split}
 \tilde u(y,s) & =\frac{u(x, t)-  (1+\vep) x_1}{\vep \alpha_{t_0,\vep}},\\
  s &=\frac{t-t_0}{\eps } .
  \end{split}
\end{equation*}
Then $\tilde u$ satisfies, for $s\in(\frac{\underline t-t_0}{\eps } , 0),$
\begin{equation}
\begin{split}
-\tilde u_s(\det D^2 \tilde u)^p &= \tilde g(y)  \ \   \text{in} \ \    T_{t_0}({ Q }),\\
\tilde u &=0  \ \  \ \ \ \ \text{on} \ \   \partial_p\big(T_{t_0}({ Q })\big) ,
\end{split}
\end{equation}
where
$$\tilde g(y)= \vep^{-p} \alpha_{t_0,\vep}^{np-1}(1+|x|^2 )^{-((n+2)p-1)/2}\approx 1 .$$
Notice that
\begin{equation*}
|\tilde u_s|=\frac{|u_t|}{\alpha_{t_0,\eps }}\approx  \sqrt{y_1^2+\eps  |\tilde y|^2}.
\end{equation*}
Hence, $ \forall \ s\in \left(\frac{\underline t-t_0}{\eps } , 0\right)$,
\begin{equation}\label{ep-ma}
\begin{split}
\det D^2 \tilde u&\approx (y_1^2+\eps  |\tilde y|^2)^{-\frac{1}{2p}} \ \   \text{in} \ \   \Sigma_{t_0}(s),\\
\tilde u &=0 \,\,\,  \quad\quad\quad\quad\quad\quad  \text{on} \ \   \partial(\Sigma_{t_0}(s)).
\end{split}  \ \
\end{equation}
 Here we denote by $\Sigma_{t_0}(s)=: T_{t_0}({ Q }(\eps  s+t_0)) \subset \mathbb R^n$ for simplicity.
Applying Alexandrov's maximum principle
\cite[Theorem 2.8]{F2017}  to \eqref{ep-ma},
we obtain
\begin{equation}\label{holder}
\begin{split}
|\tilde u(y,s)|^n\le& C d(y,\partial(\Sigma_{t_0}(s))) \int_{\Sigma_{t_0}(s)} (y_1^2+\eps  |\tilde y|^2)^{-\frac{1}{2p}}dy\\
\le & Cd(y,\partial(\Sigma_{t_0}(s))) \int_{\{|\tilde y|\le cy_1,0<y_1<c \}} y_1^{-\frac 1p}dy
\\ \le &Cd(y,\partial(\Sigma_{t_0}(s))), \ \   \text{provided} \ \   p>\frac{1}{n}.
\end{split}
\end{equation}
In the above inequality, we have used the fact that  $\tilde u$ is parabolic convex and $T_{t_0}({ Q }(t_0))$ has a good shape.
Set
$$h_0=: \sup_{T_{t_0}({ Q }(t_0))}|\tilde u(y,0 )| \approx 1.$$
By \eqref{holder}, we have
\begin{equation}
S_{{h_0/2},\tilde u}(s)\subset\subset S_{{h_0/4},\tilde u}(s)\subset\subset \Sigma_{t_0}(s)
\end{equation}
when
$S_{{h_0/2},\tilde u}(s)=\{y\in B_1(0)\ |\   \tilde u(y,s )<-h_0/2\}$
 is non-empty. Moreover
\begin{equation*}
d\big(\partial S_{{h_0/2},\tilde u}(s),\partial S_{{h_0/4},\tilde u}(s) \big),
      \ \   d\big(\partial S_{{h_0/4},\tilde u}(s),\partial (\Sigma_{t_0}(s)) \big)\ge C^{-1}.
\end{equation*}

Now, applying Lemma \ref{CH-IUMJ} to $\tilde u$ on
$\{(y,s)\ |\ \tilde u(y,s)<-h_0/4\}$
yields that
\begin{equation}\label{rupb}
\| D^2 \tilde u(\cdot,0)\|_{S_{h_0/2,\tilde u}(0)}\le C.
\end{equation}
Restricting to the $x_1$-axis,  we obtain
\begin{equation*}
\|D^2_{y_1}\tilde w\|_{S_{h_0/2,\tilde u}(0)\cap\{|\tilde x|=0\} }
 = \|D^2_{y_1}\tilde u\|_{S_{h_0/2,\tilde u}(0)\cap\{|\tilde x|=0\} }
 \le C,
\end{equation*}
where $\tilde w(y,s )=\frac{w(x,t)}{\eps  \alpha_{t_0,\vep}}$.
Scaling back to the original coordinates, we obtain
\begin{equation*}
|D^2_{x_1} w(x,t_0)|\le C \alpha_{t_0,\vep}^{n-1-1/p}  \ \   \text{at}\ x=\rho e_1
\end{equation*}
 with $\rho\approx \alpha_{t_0,\vep}$, 
which yields the first estimate in \eqref{upb}.
The second and third estimates in \eqref{upb} also follows from \eqref{rupb} by rescaling.
\end{proof}

\vspace{3mm}
\section{Bernstein theorem for a singular parabolic Monge-Amp\`ere equation}\label{s4}
\vspace{2mm}

In this section, we prove a Bernstein theorem for the following singular parabolic Monge-Amp\`ere type equation
\begin{equation}\label{blow1}
-\psi_t\det \begin{pmatrix}
   \psi_{x_nx_n}+b\frac{\psi_{x_n}}{x_n}&    \psi_{x_nx_1}  & \cdots &   \psi_{x_nx_{n-1}}\\[3pt]
   \psi_{x_nx_1}& \psi_{x_1x_1} & \cdots&\psi_{x_1x_{n-1}}\\[3pt]
   \cdots &\cdots &\cdots &\cdots \\[3pt]
  \psi_{x_nx_{n-1}} &\psi_{x_1x_{n-1}} & \cdots&\psi_{x_{n-1}x_{n-1}}
 \end{pmatrix}^p=1 \ \   \text{in} \ \   \mathbb R^{n,+} \times (-\infty,0] ,
\end{equation}
where $b$ is a constant.
Equation \eqref{blow1} arises in a blow-up argument for equation \eqref{po2}.

We have the following Bernstein theorem.

\begin{theorem}\label{thmbern}
Assume that the equation \eqref{blow1} is uniformly parabolic and $b> -1$ is a constant.
Let $\psi(x,t)\in C^{1,1}(\overline{\mathbb R^{n,+}}\times(-\infty,0])$
be a solution to \eqref{blow1}.
Assume   $\psi(0,0)=0$, $D_{x} \psi(0,0)=0$, and
$\psi_{x_n}(x',0,t)=0\ \forall\ x'\in\R^{n-1}$, $t\in(-\infty,0]$.
Then $\psi$ has the form
\begin{equation}\label{qp1}
\psi(x,t)=\frac 12 {\Small\text{$ \sum_{i,j=1}^{n-1} $}} c_{ij} x_ix_j+\frac{1}{2}c_{nn}x_n^2 - c_0 t ,
\end{equation}
where the matrix $(c_{ij})_{i,j=1}^{n-1}$ is positive definite and $c_{nn}, c_0>0$.
\end{theorem}

To prove the Bernstein theorem, we first study the linear singular operator
\begin{equation}\label{l+}
L_0 U  =:  - U_t + {\Small\text{$\sum_{i,j=1}^{n}$}} a^{ij}\p_{ij}U  + \frac{b^n}{x_n}\p_n U 
\end{equation}
with variable coefficients $a^{ij}$ and $b^n$ defined in $\mathbb R^{n,+}\times(-\infty,0]$.
Assume that $a^{ij}$ and $b^n$ satisfy
\begin{equation}\label{calpha1}
\Lambda^{-1} |\xi|^2 \le a^{ij}\xi_i\xi_j\le \Lambda |\xi|^2  \ \   \forall~\xi\in\mathbb R^n\backslash  \{0\},
\end{equation}
and
\begin{equation}\label{calpha3}
\frac{b^n}{a^{nn}} = b \in (0, \Lambda] \ \ \text{is~a~constant},
\end{equation}
for some positive constant $\Lambda$.

For a given point $p_0=(x_0,t_0)\in \mathbb R^{n,+} \times \mathbb R$, denote
\begin{equation}\label{box-r}
Q_\rho(p_0)=\{(x,t)\in \mathbb R^{n,+} \times \mathbb R \ |\ x_n >  0, |x-x_0| < \rho, t_0-\rho^2 < t \le t_0\} .
\end{equation}
If $p_0=(0, 0)$, we will write $Q_\rho(p_0)$ simply as $Q_\rho$.
Denote $\partial_p Q_\rho=\overline Q_\rho \backslash Q_\rho$ the parabolic boundary of $Q_\rho$,
and denote $\p_0 Q_\rho=\partial_p Q_\rho\cap\{x_n=0\}$,
$\p' Q_\rho= \partial_p Q_\rho \backslash \p_0 Q_\rho$.

\begin{lemma}\cite[Lemma 5.1]{L2016}\label{l-4.1} 
Assume that $a^{ij}, b^n\in C^{\infty}(\overline{Q_\rho})$
and satisfy conditions \eqref{calpha1}--\eqref{calpha3}.
Then for any function $\varphi\in C(\partial' Q_\rho)$,
there exists a unique solution $U\in C^0(\overline{Q_\rho})\cap C^2 (Q_\rho \cup \p_0 Q_\rho)$ to
\begin{equation}
\left\{ {\begin{split}
&L_0 U =0 \,\,\text{in}\,\,Q_\rho,  \\
&  U =\varphi \,\,\text{on}\,\,\p' Q_\rho, \\
& \p_n U =0 \,\,\text{on}\,\,  \p_0 Q_\rho.
\end{split}} \right.
\end{equation}
Moreover, $\sup_{Q_\rho} |U|$ is bounded by $\sup_{\partial' Q_\rho} |\varphi|$.
\end{lemma}

This lemma was proved in \cite{L2016} for operator with constant coefficients.
But the proof also works for operators with smooth coefficients.
We omit the proof here.

Next we quote a lemma from \cite{L2016,DL2003}.

\begin{lemma}\label{l-4.2}
Assume the conditions in Lemma \ref{l-4.1}.
Then there exists $\alpha=\alpha_{n, \Lambda} \in (0, 1)$ such that  for any $\rho' \in(0, \rho)$
and any smooth function $U\in C^2(\overline {Q_\rho})$,
\begin{equation}\label{a418z}
\|U\|_{ C^\alpha(Q_{\rho'})}
     \le C \Big(\sup_{Q_\rho}|U|+\Big(\int_{Q_\rho}(L_0 U)^{n+1} dx dt\Big)^{\frac1{n+1}}\Big),
\end{equation}
where the constant $C$ depends only on $n, \Lambda, \rho$ and $\rho'$.
\end{lemma}

\begin{remark}
We refer the readers to \cite[Theorem 3.1]{DL2003} or \cite[Theorem 3.3]{L2016} for the details of the proof.
\end{remark}

To apply the above lemmas to the singular parabolic Monge-Amp\`ere equation \eqref{blow1},
we make the  partial Legendre transform \cite{LS17, HTW},
\begin{equation} \label{plt}
\begin{split}
   y_n&=x_n, \ \   \\
   y'&=D_{x'}\psi, \\
   \psi^*&=x'\cdot D_{x'}\psi-\psi.
  \end{split}
\end{equation}
Then by direct computations \cite{LS17,HTW},
equation  \eqref{blow1} is changed to
\begin{equation}\label{002}
\psi^*_t \left(-\psi^*_{y_ny_n}-b\frac{\psi^*_{y_n}}{y_n}\right)^p-\left(\det D^2_{y'} \psi^*\right)^p=0  \ \   \text{in} \ \   \mathbb R^{n,+}\times(-\infty,0].
\end{equation}

\begin{lemma}\label{lemholder}
Let $\psi^* \in C^{1,1}(\overline{\mathbb R^{n,+}}\times(-\infty,0])$ be a solution to  \eqref{002}
with the constant  $b> -1$.
Assume that $\psi_{y_n}^*(y',0,t)=0 ~\forall\ y'\in\R^{n-1}$ and $t\in(-\infty,0]$,
and $D_{y'}^2 \psi^*$ is positive definite.
Then
$\frac{\psi^*_{y_n}}{y_n}\in C^{\alpha} (\overline{\mathbb R^{n,+}}\times(-\infty,0])$ for some $\alpha\in(0,1)$, and
we have the estimate
\begin{equation}\label{HE}
\Big\|\frac{\psi^*_{y_n}}{y_n}\Big\|_{C^{\alpha} (\mathbb R^{n-1}\times [0,1]\times (-\infty, 0])}  \le C
\end{equation}
for a positive constant $C$ depending only on $b, n$,
$\|\psi_t^*\|_{L^\infty(\mathbb R^{n,+}\times(-\infty,0])}$,
$\|D_y^2 \psi^*\|_{L^\infty(\mathbb R^{n,+}\times(-\infty,0])}$
and
$\|(\det D^2_{y'}\psi^*)^{-1}\|_{L^\infty(\mathbb R^{n,+}\times(-\infty,0])}$.
\end{lemma}

\begin{proof}
Differentiating equation \eqref{002}  with respect to $y_n$ yields
\begin{equation*}
\frac{\psi^*_{y_n t}}{\psi^*_t}+p\frac{\psi^*_{y_ny_ny_n}+b\left(\frac{\psi^*_{y_n}}{y_n}\right)_{y_n}}{\psi^*_{y_ny_n}+b\frac{\psi^*_{y_n}}{y_n}}=p\sum_{i,j=1}^{n-1}\Psi'^{*ij}\psi^*_{y_iy_jy_n}
\end{equation*}
where $\Psi'^{*ij}$ is the inverse matrix of $D_{y'}^2 \psi^*$.  Note that
\begin{equation*}
\left(\frac{\psi^*_{y_n}}{y_n}\right)_{y_ny_n}=\frac{\psi^*_{y_ny_ny_n}}{y_n}-2\left(\frac{\psi^*_{y_n}}{y_n}\right)_{y_n}.
\end{equation*}
Then we find that  $\Psi(y,t)=\frac{\psi^*_{y_n}}{y_n}$ satisfies
\begin{equation}\label{psi-yn}
\tilde L(\Psi)=:
  -\Psi_t  +a^{nn} \Big(\Psi_{y_ny_n}+\frac{b+2}{y_n} \Psi_{y_n}\Big)
       + {\Small\text{$  \sum_{i,j=1}^{n-1} $}}  a^{ij} \Psi_{y_iy_j}=0
\end{equation}
where
 \begin{equation*}
 (a^{ij})_{i,j=1}^{n-1}\approx I_{(n-1)\times (n-1)}, \ \   a^{nn}\approx 1,
\end{equation*}
and $I_{(n-1)\times (n-1)}$ is the unit matrix. Note that
\begin{equation*}
\Psi(y,t)=\frac{\psi^*_{y_n}}{y_n} = \int_{0}^1 \psi^*_{y_ny_n}(y',sy_n,t) ds \in L^\infty(\mathbb R^{n,+}\times(-\infty,0]).
\end{equation*}

Take a sequence of smooth functions $a^{ij}_k,a^{nn}_k, \varphi_k\in C^\infty(\overline{\mathbb R^{n,+}}\times (-\infty,0])$ such that
\begin{equation*}
(a^{ij}_k)_{i,j=1}^{n-1}\approx I_{(n-1)\times (n-1)}, \ \   a^{nn}_k\approx 1, \ \   |\varphi_k|\lesssim 1
\end{equation*}
and
\begin{equation*}
a^{ij}_k\rightarrow a^{ij}, \ \   a^{nn}_k\rightarrow a^{nn}, \ \
  \varphi_k\rightarrow \frac{\psi_{y_n}^*}{y_n}, \ \   \text{in} \ \   C^\infty_{loc}(\mathbb R^{n,+}\times (-\infty,0]).
\end{equation*}
Define
\begin{equation}\label{psi-yn-k}
\tilde L_k=: -\partial_t  +a^{nn}_k \Big(\partial_{y_ny_n}+\frac{b+2}{y_n} \partial_{y_n}\Big)
 + {\Small\text{$  \sum_{i,j=1}^{n-1} $}}  a^{ij}_k \partial_{y_iy_j}.
\end{equation}
Hence by Lemma \ref{l-4.1}, for $\rho>0$,
there exist solutions $\tilde \Psi_k\in C^0(\overline{Q_\rho})\cap C^2 (Q_\rho \cup \p_0 Q_\rho)$ of
\begin{equation}\label{lk}
\left\{ {\begin{split}
&\tilde L_k \tilde\Psi_k =0 \,\,\text{in}\,\,Q_\rho,  \\
&\tilde\Psi_k= \varphi_k \,\,\text{on}\,\,\p' Q_\rho, \\
&\p_n \tilde\Psi_k=0 \,\,\text{on}\,\,  \p_0 Q_\rho
\end{split}} \right.
\end{equation}
and $\|\tilde \Psi_k\|_{L^\infty(Q_{\rho})}$ is uniformly bounded.
By Lemma \ref{l-4.2}, for any given  $\rho'\in(0,\rho)$,
$\|\tilde \Psi_k\|_{ C^{\alpha} (\overline{Q_{\rho'}}) }$ is independent of $k$.
According to the interior regularity theory of linear parabolic equations with smooth coefficients \cite{L1996}, we have
\begin{equation*}
\|\tilde \Psi_k\|_{ C^{m+\alpha}(\overline{Q_{\rho'}} \cap \{y_n\ge\rho''\} ) }\le C_{m,\rho''}, \ \   m=2,4,6,\cdots, \ \   \rho''\in(0,\rho').
\end{equation*}
Hence, by taking a subsequence,
we may assume $\tilde \Psi_k \to \tilde \Psi $ a.e. in $\overline{Q_\rho}$
with $\tilde \Psi$ solving \eqref{psi-yn}.
Moreover, $\tilde \Psi \in C^{\alpha} (\overline{Q_{\rho'}} )\cap C^\infty(Q_{\rho}) \cap L^\infty(Q_{\rho})$.

Next, we show $\tilde \Psi=\Psi$.
Consider
\begin{equation*}
h_\eps =\tilde \Psi - \Psi + \vep y_n^{-\beta}, \ \    \text{on} \ \    \overline{Q_\rho}, \ \   \beta \in(0, b+1], \ \   \vep\in(0,1).
\end{equation*}
By the boundedness of $\tilde \Psi$ and $\Psi$, it follows that
\begin{equation*}
\lim_{y_n\rightarrow 0^+} h_\eps \rightarrow +\infty.
\end{equation*}
And $h_\eps \ge 0$ on $\partial_p Q_\rho \cap \{y_n>0\}$ follows easily by the boundary condition in \eqref{lk}.
A direct computation yields that
 \begin{equation*}
  \tilde L (h_\eps )=\vep\beta(\beta-b-1)y_n^{-\beta-2} a^{nn}\le 0, \ \   \text{in} \ \   Q_\rho.
 \end{equation*}
Then by the maximum principle, we get
$h_\eps  \ge 0$ on $ Q_\rho$.
Hence, taking $\vep \to 0$, we have $\tilde \Psi \ge  \Psi$ on $\overline{Q_\rho}$.
Similarly, we have $\tilde \Psi \le  \Psi$.
Therefore, we have $\tilde \Psi=\Psi$,
which yields $\frac{\psi^*_{y_n}}{y_n}\in C^{\alpha}(\overline{\mathbb R^{n,+}}\times(-\infty,0])$
 and estimate \eqref{HE}.
\end{proof}

\begin{lemma}\label{lemconti}
Assume the assumptions of Theorem \ref{thmbern}.
Then  $\psi\in C^{2+\alpha} (\overline{\mathbb R^{n,+}}\times(-\infty,0])$ for some $\alpha\in (0, 1)$.
\end{lemma}

\begin{proof}
Let $\psi^*$ be the partial Legendre transform of $\psi$.
Then $\psi^*$ satisfies equation \eqref{002} and  the assumptions of Lemma \ref{lemholder}.
Hence by Lemma \ref{lemholder}, $\frac{\psi^*_{y_n}}{y_n}\in C^{\alpha}(\overline{\mathbb R^{n,+}}\times(-\infty,0])$.
Recall that
$\frac{\psi_{x_n}(x,t)}{x_n}=-\frac{\psi^*_{y_n}(y,t)}{y_n}$.
We therefore have
\begin{equation*}
\begin{split}
\Big|\frac{\psi_{x_n}(x,t)}{x_n}-\frac{\psi_{x_n}(\tilde x,\tilde t)}{\tilde x_n}\Big|
        =\Big|\frac{\psi^*_{y_n}(y,t)}{y_n}-\frac{\psi^*_{y_n}(\tilde y,\tilde t)}{\tilde y_n}\Big|
\le C\left(|y-\tilde y|^\alpha+|t-\tilde t|^{\frac \alpha2}\right).
\end{split}
\end{equation*}
By the partial Legendre transform \eqref{plt}, $y_n=x_n,$ $y'= D_{x'}\psi$.
It follows that
\begin{equation*}
|y'-\tilde y'|=| D_{x'} \psi(x,t)- D_{x'}\psi(\tilde x,t)|\le \|D^2\psi\|_{L^\infty(\mathbb R^{n,+}\times(-\infty,0])} |x-\tilde x|.
\end{equation*}
Hence $\frac{\psi_{x_n}}{x_n}\in C^{\alpha} (\overline{\mathbb R^{n,+}}\times(-\infty,0])$ and we have the estimate
\begin{equation}\label{dptc}
\Big\|\frac{\psi_{x_n}}{x_n}\Big\|_{C^{\alpha}(\mathbb R^{n-1}\times[0,1]\times(-\infty,0])}\le C,
\end{equation}
where $C$ depends only on
$n, b$, $\|\psi_t\|_{L^\infty(\mathbb R^{n,+}\times(-\infty,0])}$ and $\|D_x^2 \psi\|_{L^\infty(\mathbb R^{n,+}\times(-\infty,0])}$.

Then, we make an even extension of $\psi(x,t)$ with respect to the variable $x_n$ and
still denote it by $\psi(x,t)$.
We write equation \eqref{blow1} in the form
\begin{equation*}
\mathcal F(\psi_t, \frac{\psi_{x_n}}{x_n},D^2_x\psi) = -\psi_t- \frac{1}{(\det W_{\psi})^p}=0,
\end{equation*}
where $W_{\psi}=W_{\psi}(\frac{\psi_{x_n}}{x_n},D^2_x\psi)$ denotes the matrix in equation \eqref{blow1}.
Here, we can regard $\frac{\psi_{x_n}}{x_n}$ as a known function, which is H\"older continuous.
By our assumptions,  $\mathcal F$ is  fully nonlinear, uniformly parabolic
and the coefficients in $\mathcal F$ are $C^{\alpha}$ smooth.
Since $\mathcal F$ is concave with respect to $D^2\psi$, then
by the $C^{2,\alpha}$ regularity theory of fully nonlinear parabolic equations \cite{W1992-1,W1992-2},
we conclude that $\psi\in C^{2+\alpha} (\R^n\times (-\infty,0])$.
\end{proof}

\noindent
{\bf Proof of Theorem \ref{thmbern}:}\
Let $\psi$ be the solution in Theorem \ref{thmbern}.   Let
\begin{equation}
\psi^{m}(x,t)=\frac{\psi(mx,m^2t)}{m^2}, \ \   m=1,2,\cdots
\end{equation}
be a blow-down sequence of $\psi$.
Since \eqref{blow1} is uniformly parabolic for $\psi$, it is also uniformly parabolic for $\psi^m$ with the same bounded constants.
This implies that there is a constant $C>0$,
independent of $m$, such that
\begin{equation*} \label{cmc}
C^{-1} I_{n\times n}\le W_{\psi^m} \le C I_{n\times n},
\end{equation*}
where $I_{n\times n}$ is the unit matrix and $W_{\psi^m}$ denotes the matrix in equation \eqref{blow1}.
Note that $\psi^{m}$ satisfies the conditions in Lemma \ref{lemconti}, uniformly in $m$.
Thus,  by Lemma \ref{lemconti} we have
\begin{equation}
\begin{split}
&|D_x^2 \psi(x,t)-D_x^2 \psi(0,0)|+|D_t \psi(x,t)-D_t\psi(0,0)|\\
=&\lim_{m\rightarrow +\infty}\left(\big|D_x^2 \psi^m\big( {\Small\text{$\frac xm$}},
   {\Small\text{$\frac t{m^2}$}} \big)-D_x^2 \psi^m(0,0)\big|+\big|D_t \psi^m\big( {\Small\text{$\frac xm$}}, {\Small\text{$\frac t{m^2}$}} \big)-D_t \psi^m(0,0)\big|\right)=0
\end{split}
\end{equation}
for any given point $(x,t)\in \R^{n,+}\times (-\infty,0]$.
Therefore, $\psi$ is the form
\begin{equation}\label{qp2}
\psi(x,t)=\frac 12 {\Small\text{$ \sum_{i,j=1}^{n} $}} c_{ij} x_ix_j - c_0 t.
\end{equation}
By the assumption $\psi_{x_n}(x',0,t)=0\ \forall\ x'\in\R^{n-1}$ and $t\in(-\infty,0]$,
we have $c_{in}=0$, for $i=1,\cdots,n-1$, in the polynomial \eqref{qp2}.
\qed

In the elliptic case, the Bernstein theorem was obtained in  \cite{HTW},
where a counter-example was also given when the condition in Theorem \ref{thmbern}
is not satisfied.

We also point out that the Bernstein theorem for parabolic Monge-Amp\`ere equation
$$-\psi_t\det D^2\psi =1\ \ \text{in}\  \R^n \times (-\infty,0]$$
was obtained by Guti\'errez and Huang \cite{GH1998},
under the assumption $C_1\le -\psi_t \le C_2$.

\vspace{3mm}
\section{Estimates for the modulus of continuity of $\zeta_t$ and $D^2\zeta$}\label{s5}
\vspace{2mm}

Let $u(\cdot, t)$ be the Legendre transform of $v(\cdot, t)$.
By assumption \eqref{rho-0}, we have $u(0,t)=0$ and  $u(x,t)>\rho_0 |x|\ \forall\ x\ne 0$, $t\in[0,T]$.
Let $\phi(x,t)$ be the tangential cone of $u$ at $(0, t)$ and $w=u-\phi$.

To study the regularity of $u$,
we introduce the spherical coordinates $(\theta, r)$,
where $\theta$ is an orthonormal frame on $\mathbb S^{n-1}$.
Let
\begin{equation}\label{zeta0}
\zeta=\frac {u(\theta,r,t)}r,
\end{equation}
where $r=|x|\in(0,1]$.

As in \cite{HTW} we can verify

\begin{lemma}\label{L4.1}
The function  $\zeta(\theta,r,t)$ satisfies the parabolic Monge-Amp\`ere type equation
\begin{equation}\label{po1r}
-\zeta_t\det \begin{pmatrix}
  \displaystyle\frac{r^{\frac1p} \zeta_{rr}}{r^{n-2}}+\frac{2r^{\frac1p} \zeta_r}{r^{n-1}} & \displaystyle\frac{r^{\frac1{2p}} \zeta_{r\theta_1}}{r^{\frac{n-2}{2}}} &\cdots&
                    \displaystyle\frac{r^{\frac{1}{2p}} \zeta_{r \theta_{n-1}}}{r^{\frac{n-2}{2}}}\\
  \displaystyle\frac{r^{\frac1{2p}} \zeta_{r\theta_1}}{r^{\frac{n-2}{2}}}& \zeta_{\theta_1\theta_1}+\zeta+r\zeta_r & \cdots&
                         \zeta_{\theta_1\theta_{n-1}}\\
        \cdots &\cdots &\cdots &\cdots  \\
  \displaystyle\frac{r^{\frac1{2p}} \zeta_{r\theta_{n-1}}}{r^{\frac{n-2}{2}}} &\zeta_{\theta_1\theta_{n-1}} & \cdots&
                                   \zeta_{\theta_{n-1}\theta_{n-1}}+\zeta+r\zeta_r
\end{pmatrix}^p =F(r),
\end{equation}
where $F(r)=(1+r^2)^{-\frac{(n+2)p-1}{2}}$.
\end{lemma}

By Lemma \ref{L2.4}, the matrix in \eqref{po1r} is uniformly bounded.
We can make \eqref{po1r} uniformly parabolic.
Let
$$s=r^{\frac {\sigma_p} 2}.$$
Then
\begin{equation}\label{zes}
{\begin{split}
\zeta_s &=\frac{2}{\sigma_p}r^{1-\frac {\sigma_p}{2}}\zeta_r,\\
\zeta_{s\theta}&=\frac{2}{\sigma_p}r^{1-\frac {\sigma_p}{2}}\zeta_{r\theta},\\
\zeta_{ss} & =\frac{4}{\sigma_p^2 }r^{2-\sigma_p} \zeta_{rr} -\frac{2(\sigma_p-2)}{\sigma_p^2 }r^{1-\sigma_p}\zeta_r.
\end{split}}
\end{equation}
Hence by   \eqref{zes} and Lemma \ref{L2.4}, we have

\begin{corollary}\label{C5.1}
As a function of $\theta$, $s$ and $t$, $\zeta$ satisfies $\zeta_{s}(\theta,0,t)=0$ 
  $\forall~ \theta\in \mathbb S^{n-1}$ and $t\in[0,T]$, and $\zeta_t,D_{\theta,s}^2 \zeta \in L^\infty(\mathbb S^{n-1}\times  [0, 1]\times [0,T])$.
\end{corollary}

By  \eqref {zes}, equation \eqref{po1r} changes to
\begin{equation}\label{po2}
-\zeta_t\det \begin{pmatrix}
 \zeta_{ss}+\frac{2+\sigma_p}{\sigma_p}\frac{\zeta_s}{s}   &   \zeta_{s\theta_{1}}&\cdots& \zeta_{s\theta_{n-1}}\\[3pt]
 \zeta_{s\theta_1} & \zeta_{\theta_1\theta_1}+\zeta+\frac {\sigma_p}2 s\zeta_s & \cdots&\zeta_{\theta_1\theta_{n-1}}\\[3pt]
                    \cdots &\cdots &\cdots &\cdots  \\[3pt]
 \zeta_{s\theta_{n-1}} &\zeta_{\theta_1\theta_ {n-1}} & \cdots&\zeta_{\theta_{n-1}\theta_{n-1}}+\zeta+\frac {\sigma_p}2 s\zeta_s
\end{pmatrix}^p=\bar {F} (s),
\end{equation}
where $\bar {F} (s)= 4^p\sigma_p^{-2p} \big(1+s^{\frac4{\sigma_p}}\big)^{-\frac{(n+2)p-1}2}$.

\begin{lemma}\label{unif-ellip}
Equation \eqref{po2} is uniformly parabolic.
\end{lemma}

\begin{proof}
This follows directly from Corollary \ref{ut-sup}, Lemma \ref{ut-sub} and Corollary \ref{C5.1}.
\end{proof}

The main result of this section is the following theorem.

\begin{theorem}\label{thm2}
Let $\zeta(\theta,s,t)\in C^{1,1} (\mathbb S^{n-1}\times [0,1]\times [0,T])$
be a solution to \eqref{po2} with $\zeta_{s}(\theta,0,t)=0$ $\forall~\theta\in\mathbb S^{n-1}$ and $t\in [0,T]$.
Then $\zeta(\theta,s,t)\in C^2 (\mathbb S^{n-1}\times [0,1]\times (0,T])$.
\end{theorem}

The proof of Theorem \ref{thm2} uses similar ideas as in \cite[Theorem 4.2]{HTW},
where the continuity of the second derivatives was obtained for the Monge-Amp\`ere obstacle problem.
The proof is divided into four lemmas.

\begin{lemma}\label{lemconti1}
Let $\zeta(\theta,s,t)$ be as in Theorem \ref{thm2}. Then $\zeta_{s\theta}\in C(\mathbb S^{n-1}\times[0,1]\times (0,T])$.
\end{lemma}

\begin{proof}
Since $\zeta\in C^{\infty}(\mathbb S^{n-1}\times(0,1]\times (0,T])$, it is enough to show $\zeta_{s\theta}$ is continuous on $\{s=0\}$.
If the lemma is not true,
there exists  a sequence of points  $(\theta^k,s^k,t^k)\rightarrow (\theta^*, 0, t^*)$
such that
\begin{equation}\label{contra1}
\lim_{k\rightarrow +\infty}|\zeta_{s\theta}(\theta^k,s^k,t^k)|\ge \eps _0,\quad s^k>0
\end{equation}
for a constant $\eps _0>0$ and $t^*\in (0,T]$.
In the following, we choose $\theta^*=0$ as the origin of a local coordinates on the unit sphere.

We then make the coordinate transform
\begin{equation}\label{coortran1}
\begin{split}
\theta &=\lambda_k \varphi+\theta^k,\\
s &=\lambda_k \tau,\\
t&=\lambda_k^2\sigma + t^k 
\end{split}
\end{equation}
where $\lambda_k=s^k$.
Let
\begin{equation}\label{coortran2}
\tilde \zeta^k(\varphi,\tau,\sigma)=\frac{\zeta(\theta,s,t)- \zeta(\theta^k,0,t^k)- D_\theta \zeta(\theta^k,0,t^k)\cdot (\theta-\theta^k)}{\lambda_k^2}.
\end{equation}
Then by the estimates in Sections \ref{s2} and \ref{s3}, we have
\begin{equation}
C^{-1} \le -\tilde \zeta^k_\sigma(\varphi,\tau,\sigma) \le C 
\end{equation}
and
\begin{equation}
|\tilde \zeta^k(\varphi,\tau,\sigma)|\le C(\tau^2+|\varphi|^2 + |\sigma|)
\end{equation}
for a constant $C>0$ independent of $k$.
Moreover, by \eqref{po2}, $\tilde \zeta^k(\varphi,\tau,\sigma)$ satisfies the equation
\begin{equation}\label{po3}
-\tilde \zeta^k_\sigma \det \begin{pmatrix}
 \tilde \zeta^k_{\tau\tau}+ \frac{2+\sigma_p}{\sigma_p}\frac{\tilde \zeta^k_\tau}{\tau}
&\tilde \zeta^k_{\tau\varphi_{1}}&\cdots& \tilde \zeta^k_{\tau\varphi_{n-1}}\\[5pt]
\tilde  \zeta^k_{\tau\varphi_1}& \tilde \zeta^k_{\varphi_1\varphi_1}+h^k(\varphi,\tau)
                  & \cdots&\tilde \zeta^k_{\varphi_1\varphi_{n-1}}\\[3pt]
    \cdots &\cdots &\cdots &\cdots  \\[3pt]
\tilde  \zeta^k_{\tau\varphi_{n-1}} &\tilde \zeta^k_{\varphi_1\varphi_{n-1}} & \cdots&\tilde \zeta^k_{\varphi_{n-1}\varphi_{n-1}}+h^k(\varphi,\tau)
\end{pmatrix}^p=\bar {F}(\lambda_k\tau),
\end{equation}
where
$$h^k=\lambda_k^2(\tilde \zeta^k+\frac{\sigma_p}2 \tau\tilde \zeta^k_\tau)+\zeta(\theta^k,0,t^k) 
       +\lambda_k D_{\theta} \zeta(\theta^k,0,t^k)\cdot \varphi\to \zeta(0,0,t^*)\ \ \text{as}\  k\to\infty . $$

Denote $\widetilde W_{k}$ the matrix in equation \eqref{po3}.
We can write \eqref{po3} as a general fully nonlinear parabolic equation of the form
\begin{equation}\label{Fk}
\mathcal F_k(\varphi,\tau, \tilde \zeta^k,\tilde \zeta^k_\sigma, D\tilde \zeta^k, D^2\tilde \zeta^k)
=:  -\tilde \zeta^k_\sigma- \bar {F}(\lambda_k\tau) \frac{1}{(\det \widetilde W_{k})^p}=0.
\end{equation}
According to Lemma \ref{unif-ellip}, $\mathcal F_k$ is  uniformly parabolic.
Moreover, $\mathcal F_k$ is concave with respect to its variables $D^2 \tilde \zeta^k$
and is  smooth in all its arguments for $\tau>0$.
Hence by Krylov's regularity theory \cite{K1983},  we have
\begin{equation}\label{con-a2}
\|\tilde \zeta^{k}\|_{ C^{4+\alpha}(\overline{Q})}\le C_{Q} \ \   \forall\ Q\subset\subset \mathbb R^{n,+}\times(-\infty,0],
\end{equation}
where the constant $C_Q$ is independent of $k$.
By passing to a subsequence, we have
\begin{equation*}
\tilde \zeta^k(\varphi,\tau,\sigma)\rightarrow \bar \zeta(\varphi,\tau,\sigma)
   \ \   \text{in} \ \
   C^{4+\alpha}_{loc}(\mathbb R^{n,+}\times(-\infty,0])\cap C^{2-\eps}_{loc}(\overline{\mathbb R^{n,+}}\times(-\infty,0])
\end{equation*}
for a function
$\bar \zeta \in C^{4+\alpha}_{loc}(\mathbb R^{n,+}\times(-\infty,0])\cap C^{1,1}_{loc}(\overline{\mathbb R^{n,+}}\times(-\infty,0])$,
where $\vep\in(0,1)$ is any small constant.
Hence $\bar \zeta$ satisfies equation \eqref{blow1}
with $b=\frac{2+\sigma_p}{\sigma_p}$ and variables $x'=\varphi$, $x_n=\tau$, $t=\sigma$.
By Theorem \ref{thmbern},   $\bar \zeta$ is of the form
\begin{equation*} \label{bar-zeta}
\bar \zeta(\varphi,\tau,\sigma)=\frac{1}{2}c_{nn}\tau^2+\frac{1}{2} {\Small\text{$  \sum_{i,j=1}^{n-1} $}} c_{ij}\varphi_i\varphi_j-c_0\sigma.
\end{equation*}
Hence the mixed derivatives $\bar \zeta_{\tau\varphi}(0',1,\sigma)=0$ for all $\sigma\in(-\infty,0]$.
By the interior regularity for equation \eqref{Fk}  \cite{K1983}, it implies that
\begin{equation}\label{zetast}
\lim_{k\rightarrow +\infty}\zeta_{s\theta}(\theta^k,s^k,t^k)
   =\lim_{k\rightarrow +\infty}\tilde \zeta^k_{\tau\varphi}(0',1,0)=\bar \zeta_{\tau\varphi}(0',1,0)=0.
\end{equation}
We reach a contradiction with \eqref{contra1}.
The lemma is thus proved.
\end{proof}

\begin{lemma}\label{lemconti2}
Let $\zeta(\theta,s,t)$ be as in Theorem \ref{thm2}.
Then $\zeta_{ss}\in C(\mathbb S^{n-1}\times[0,1]\times (0,T])$.
\end{lemma}

\begin{proof}
It is enough to prove the continuity of $\zeta_{ss}$ on $\{s=0\}$,  i.e.
$\lim\limits_{(\theta,s,t)\rightarrow (0,0^+,t^*)}\zeta_{s s}(\theta,s,t)$ exists.  For simplicity, in the following, we only consider $t^*=T$.
By Lemma \ref{unif-ellip}, $ \zeta_{s s}(\theta, s,t)$ is uniformly bounded.
Hence there is a sub-sequence  $s^k\rightarrow 0$ such that
$\zeta_{ss}(0, s^k,T)$ is convergent.
We introduce the coordinates $(\varphi,\tau,\sigma)$ and function $\tilde \zeta^k$
as in \eqref{coortran1} and \eqref{coortran2}, with $\theta^k=0$, $t^k=T$.

By the proof of Lemma \ref{lemconti1}, we have
\begin{equation} \label{zeta12}
\tilde \zeta^k(\varphi,\tau,\sigma)\rightarrow \frac{1}{2}c_{nn}\tau^2+\frac{1}{2} {\Small\text{$  \sum_{i,j=1}^{n-1} $}} c_{ij}\varphi_i\varphi_j - c_0\sigma
\end{equation}
in $C^{4+\alpha}_{loc}(\mathbb R^{n,+}\times(-\infty,0])\cap C^{2-\eps}_{loc}(\overline{\mathbb R^{n,+}}\times(-\infty,0])$.
\begin{remark}
For general $t^*\in (0,T)$, it is enough to replace $\mathbb R^{n,+}\times(-\infty,0]$ by $\mathbb R^{n,+}\times(-\infty,1]$ in the proof.
\end{remark}
By the convergence \eqref{zeta12} and the interior regularity of equation \eqref{po3},
we can choose a subsequence, such that 
\begin{equation}\label{zeta-c00}
\Big|\frac{\tilde \zeta^k_\tau(\varphi,\tau,\sigma)}{\tau}-c_{nn}\Big| \le \frac 1{2k}  \ \   \text{in}\  Q_k
\end{equation}
where 
$$Q_k=:\big\{(\varphi,\tau,\sigma)\ | \  |\varphi|\le 1,\	\frac 1k\le \tau\le 1, -1\le \sigma\le 0\big\}.$$
Scaling back to $\zeta(\theta,s,t)$, we obtain
\begin{equation}\label{aym-aa2}
\Big|\frac{\zeta_s(\theta,s,t)}{s}-c_{nn}\Big| \le \frac 1{2k}  \ \  \ \text{in}\ \Sigma_k,
\end{equation}
where 
$$
\Sigma_k=:\big\{(\theta,s,t) \ |\ |\theta|\le \lambda_k,\	\frac {\lambda_k}k\le s\le \lambda_k, -\lambda_k^2\le t-T\le 0 \big\}.
$$
Let 
$${\mathfrak r}=\frac{s^2}{4} .$$
The above estimate implies that
\begin{equation}\label{aym-a2}
\left|\zeta_{\mathfrak r}-2c_{nn}\right| \le \frac 1k
 \ \   \text{in}\ \big\{(\theta,{\mathfrak r},t)\ |\ |\theta|\le \lambda_k,\	\frac {\lambda^2_k}{4k^2}\le {\mathfrak r}\le \frac{\lambda_k^2}{4}, -\lambda_k^2\le t-T\le 0 \big\},
\end{equation}
and equation \eqref{po2} changes to
\begin{equation}\label{max2}
-\zeta_t \det \begin{pmatrix}
   {\mathfrak r}\zeta_{{\mathfrak r}{\mathfrak r}}+\frac{1+\sigma_p}{\sigma_p}\zeta_{\mathfrak r}&    \zeta_{{\mathfrak r}\theta_{1}}&\cdots&    \zeta_{{\mathfrak r}\theta_{n-1}}\\[5pt]
   {\mathfrak r}\zeta_{{\mathfrak r}\theta_1}& \zeta_{\theta_1\theta_1}+\zeta+ \sigma_p{\mathfrak r}\zeta_{\mathfrak r} & \cdots&\zeta_{\theta_1\theta_{n-1}}\\[3pt]
      \cdots &\cdots &\cdots &\cdots   \\[3pt]
   {\mathfrak r}\zeta_{{\mathfrak r}\theta_{n-1}} &\zeta_{\theta_1\theta_{n-1}} & \cdots&\zeta_{\theta_{n-1}\theta_{n-1}}+\zeta+ \sigma_p{\mathfrak r}\zeta_{\mathfrak r}
\end{pmatrix}^p= \tilde {F},
\end{equation}
where
$$\tilde {F}({\mathfrak r})=\bar {F}(2{\mathfrak r}^{1/2})= 4^p \sigma_p^{-2p}\big(1+(4{\mathfrak r})^{\frac{2}{\sigma_p}}\big)^{- \frac{(n+2)p-1}2} .$$
The coefficient ${\mathfrak r}$ in the first column in \eqref{max2} is due to $\zeta_{s\theta} ={\mathfrak r}^{1/2} \zeta_{{\mathfrak r}\theta}$.

For convenience, we denote $W=\{W_{ij}\}_{i,j=1}^n$ by the matrix in  equation \eqref{max2} and rewrite the equation as
\begin{equation}\label{linear-l}
\log (-\zeta_t) +p\log (\det W) =\log  \tilde {F}.
\end{equation}
Differentiating in ${\mathfrak r}$, we have
$$\frac{\zeta_{t{\mathfrak r}}}{\zeta_t}+pW^{ij}\partial_{\mathfrak r} W_{ij}-\frac{\tilde {F}'}{\tilde {F}} = 0, $$
where $\{W^{ij}\}$ is the inverse of $\{W_{ij}\}$.
Denoting $V=\zeta_{\mathfrak r}$,
we get
\begin{equation}\label{max3}
\begin{split}
\mathcal L(V)=: & - V_t + \tilde a^{nn}\big({\mathfrak r}V_{{\mathfrak r}{\mathfrak r}}
    +  (2+{\small\text{$  \frac{1}{\sigma_p} $}}) V_{{\mathfrak r}} \big)+  {\Small\text{$ \sum_{i,j=1}^{n-1} $}}\tilde a^{ij}V_{\theta_i\theta_j}\\
 & \ \   +  {\Small\text{$ \sum_{i=1}^{n-1}$}} \tilde a^{ni} {\mathfrak r}^{1/2} V_{{\mathfrak r}\theta_i}
    +  {\Small\text{$ \sum_{i,j=1}^{n-1} $}} \zeta_{{\mathfrak r}\theta_i}b^{ij} V_{\theta_j}=\bar h \tilde{F}'+ \tilde h
\end{split}
\end{equation}
where $ \tilde a^{ij},b^{ij} $, $\bar h$ and $\tilde h$ are  continuous functions of the elements in the matrix in \eqref{max2},
namely ${\mathfrak r},\zeta, \zeta_t, \zeta_{\mathfrak r}, \zeta_{\theta_i\theta_j}, {\mathfrak r}\zeta_{{\mathfrak r}{\mathfrak r}},{\mathfrak r}^{1/2}\zeta_{{\mathfrak r}\theta_i},{\mathfrak r}\zeta_{{\mathfrak r}\theta_i}\zeta_{{\mathfrak r}\theta_j}$.
From the assumptions in Theorem \ref{thm2}, all the elements are uniformly bounded, namely
\begin{equation}\label{aym-a1}
|\zeta_t|+|\zeta_{\mathfrak r}|+|{\mathfrak r}\zeta_{{\mathfrak r}{\mathfrak r}}|+|{\mathfrak r}^{1/2}\zeta_{{\mathfrak r}\theta_i}|
        + |{\mathfrak r}\zeta_{{\mathfrak r}\theta_i} \zeta_{{\mathfrak r}\theta_j}|+| D^2_{\theta}\zeta|\le C
         \ \   \forall\ (\theta,{\mathfrak r},t)\in \mathbb S^{n-1}\times [0,   {\Small\text{$\frac 1{16}$}} ]\times[-1,0] .
\end{equation}
It implies that $\tilde a^{ij},b^{ij}, b^j,\bar h,\tilde h$ are uniformly bounded and $(\tilde a^{ij})_{i,j=1}^n$ is uniformly elliptic.
In fact, by rescaling and the interior estimates for uniformly parabolic equations, one can also get
\begin{equation}\label{asym-a-interior}
|D_{t,\mathfrak r}^k D_{\theta}^l \zeta|\le C_{k,l} {\mathfrak r}^{1-k-\frac l2},\quad 
    \forall\ (\theta,{\mathfrak r},t)\in \mathbb S^{n-1}\times [0,   {\Small\text{$\frac 1{16}$}} ]\times[-1,0],\quad k,l=1,2,\cdots .
\end{equation}

Let $\varphi$ be a cut-off function of $(\theta,t)$ such that $0\le \varphi\le 1$ and
\begin{equation*}
{\begin{split}
 \varphi(\theta,t) & \equiv 1\ \ \text{when} \  |\theta|\le  {\Small\text{$ \frac{1}{2} $}}~ \text{and}~ {\Small|t-T|\le  \text{$ \frac{1}{2} $}};\\
 \varphi & \equiv 0\ \ \text{when}\ |\theta|>1 ~ \text{and}~|t-T|>1.
 \end{split}}
\end{equation*}
Denote $\varphi_k(\theta,t)=\varphi\big(\frac{\theta}{\lambda_k}, \frac{t-T}{\lambda_k^2}+T\big)$ and  $\widehat V^k=\varphi_k V$.
Then   $\widehat V^k$ satisfies
\begin{equation*}
\mathcal L(\widehat V^k)=\varphi_k (\bar h {\tilde F}'+ \tilde h)-[\varphi_k,\mathcal L] V=: \widehat h^k,
\end{equation*}
where
$$ [\varphi_k,\mathcal L]  V=\varphi_k\mathcal L V -\mathcal L(\varphi_k V) . $$
By \eqref{aym-a1} and \eqref{asym-a-interior}, we have
\begin{equation}\label{hkc}
{\begin{split}
|\widehat h^k| &\le C(1+\lambda_k^{-2}+\lambda_k^{-1}{\mathfrak r}^{-\frac 12}+{\mathfrak r}^{\frac{2}{\sigma_p}-1})\\
& \le C(\lambda_k^{-1}{\mathfrak r}^{-\frac 12}+{\mathfrak r}^{\frac{2}{\sigma_p}-1})
          \ \ \  \text{when}\ \ 0<{\mathfrak r}<\lambda_k^2 ,
          \end{split}}
\end{equation}
where $C$ is a positive constant independent of $k$.

Denote   $\delta_{k, -1}  =\frac14$,  $ \delta_{k,0}  =\frac{1}{4k^2}$. Let
\begin{equation}\label{dekl}
\bigg\{ {\begin{split}
 & \delta_{k,\gamma+1} = (\delta_{k,\gamma})^{1+ \frac{\bar a}2\sigma_p},\\
 & \eps _{k,\gamma} =C_1 \delta_{k,\gamma}^{\bar a+\frac1{\sigma_p}}\lambda_k^{\frac 2{\sigma_p}},
  \end{split}} \ \ \  \ \   \gamma=0,1,2,\cdots ,
\end{equation}
where $C_1$ is a sufficiently large constant, $\bar a=:\min \{\frac 12, \frac 2{\sigma_p}\}$.
%Then a similar argument as in \cite[Lemma 4.4]{HTW} yields the present lemma.

{\it Claim}: For any given $\gamma\ge 1$, we have
\begin{equation}\label{claim-conti1}
\begin{split}
|\widehat V^k-2\varphi_k c_{nn}|\le  {\small\text{$ \frac 1k $}} +C_1  {\Small\text{$ \sum_{l=0}^{\gamma-1} $}}    \delta_{k,l}^{\frac {\bar a}2}
   \end{split}
\end{equation}
when $\ |\theta|\le \lambda_k,\  \delta_{k,\gamma}\lambda_k^2\le \mathfrak r\le \delta_{k,\gamma-1}\lambda_k^2,\	-\lambda_k^2\le t-T\le 0$.

We  prove \eqref{claim-conti1} by induction.
By  \eqref{aym-a2},    \eqref{claim-conti1} holds  for  $\gamma=0$.
Assuming that \eqref{claim-conti1} holds for $\gamma$, we prove that it holds for $\gamma+1$.
We introduce the auxiliary functions
\begin{equation}\label{sigmapm}
\sigma_{k,\gamma}^{\pm}(\theta,\mathfrak r,t)
  =2\varphi_k c_{nn}\pm\Big(\frac 1k
    + C_1{\Small\text{$ \sum_{l=0}^{\gamma-1} $}} \delta_{k,l}^{\frac {\bar a}2}+\varepsilon_{k,\gamma} \mathfrak r^{-\frac 1{\sigma_p}}\Big).
\end{equation}
By our choice of $\varphi_k$ and \eqref{hkc},  we have
\begin{equation}
\begin{split}
\mathcal L(\widehat V^k-\sigma_{k,\gamma}^{+})
 & =  \widehat h_k-2c_{nn}\mathcal L(\varphi_k)+\tilde a^{nn}\frac{\varepsilon_{k,\gamma}}{\sigma_p}{\mathfrak r}^{-1-\frac 1{\sigma_p}}\\
 &\ge  -C(\lambda_k^{-1}\mathfrak r^{-\frac 12}+\mathfrak r^{\frac{2}{\sigma_p}-1})+\frac 1C \varepsilon_{k,\gamma}\mathfrak r^{-\frac{1}{\sigma_p}-1}
   \\
  & > 0 \ \ \ \ \
     \quad \text{when}\ 0< \mathfrak r\le \delta_{k,\gamma}\lambda_k^2
\end{split}
\end{equation}
if the constant $C_1$ in \eqref{dekl} is chosen large.
By our induction assumptions, we have
\begin{equation*}
\begin{split}
     &\widehat V^k-\sigma_{k,\gamma}^{+}\le 0\quad 
              \text{if}\  |\theta|\le \lambda_k,\  \mathfrak r=\delta_{k,\gamma}\lambda_k^2,\  -\lambda_k^2\le t-T\le 0,\\
     &\widehat V^k-\sigma_{k,\gamma}^{+} 
         =-\Big(\frac 1k+ C_1{\Small\text{$  \sum_{l=0}^{\gamma-1} $}}  \delta_{k,l}^{\frac {\bar a}2}+\varepsilon_{k,\gamma} {\mathfrak r}^{-\frac 1{\sigma_p}}\Big)<0
           \quad \text{if}\  |\theta|=\lambda_k,\	\text{or}\	t-T=-\lambda_k^2,\\
\ \hskip15pt \ & \lim\sup_{\mathfrak r\rightarrow 0^+}(\widehat V^k-\sigma_{k,\gamma}^{+})<0.
\end{split}
\end{equation*}
By the maximum principle, it follows that
\begin{equation*}
      \widehat V^k-\sigma_{k,\gamma}^+\le 0
  \quad \text{if}\  |\theta|\le \lambda_k,\    0< \mathfrak r\le \delta_{k,\gamma}\lambda_k^2,\  -\lambda_k^2\le t-T\le 0.
\end{equation*}
Similarly, we have
\begin{equation*}
    \widehat V^k-\sigma_{k,\gamma}^-\ge  0
   \quad \text{if}\ |\theta|\le \lambda_k,\    0< \mathfrak r\le \delta_{k,\gamma}\lambda_k^2,\  -\lambda_k^2\le t-T\le 0.
\end{equation*}
For $|\theta|\le \lambda_k,\ \delta_{k,\gamma+1}\lambda_k^2\le \mathfrak r\le \delta_{k,\gamma} \lambda_k^2$, $-\lambda_k^2\le t-T\le 0$, we obtain
\begin{equation*}
\begin{split}
|\widehat V^k-2c_{nn}\varphi_k|
     &\le \frac 1k+C_1 {\Small\text{$ \sum_{l=0}^{\gamma-1} $}} \delta_{k,l}^{\frac {\bar a}2}
                +\varepsilon_{k,\gamma}\mathfrak r^{-\frac 1{\sigma_p}}\\
     &\le  \frac 1k+C_1 {\Small\text{$ \sum_{l=0}^{\gamma-1} $}} \delta_{k,l}^{\frac {\bar a}2}
                +\varepsilon_{k,\gamma}\delta_{k,\gamma+1}^{-\frac 1{\sigma_p}}\lambda_k^{-\frac 2{\sigma_p}}\\
     &\le  \frac 1k+C_1{\Small\text{$ \sum_{l=0}^{\gamma} $}} \delta_{k,l}^{\frac {\bar a}2}.
\end{split}
\end{equation*}
The claim \eqref{claim-conti1} is proved.

For any  point $(\hat \theta,\hat {\mathfrak r},\hat t)$ near $(0, 0,T)$ with $\hat {\mathfrak r}>0$,
we can choose $k>0$ such that
$$(\hat \theta,\hat{\mathfrak r},\hat t)\in 
      \big\{(\theta,\mathfrak r,t):\ |\theta|\le\frac{\lambda_k}{2}, 0<\mathfrak r\le \frac{\lambda_k^2}{4},-\lambda_k^2\le t-T\le 0\big\} .$$
We then choose $\gamma\ge 0$ such that $\delta_{k,\gamma+1}\lambda_k^2\le \hat {\mathfrak r}\le \delta_{k,\gamma}\lambda_k^2$.
Hence we have
\begin{equation}\label{est1}
{\begin{split}
|\widehat V_k-2\varphi_k c_{nn}|
  & \le \frac 1k+C_1 {\Small\text{$ \sum_{l=0}^{\gamma} $}} \delta_{k,l}^{\frac {\bar a}2}\\
  &  \le  \frac 1k+C_1 {\Small\text{$ \sum_{l=0}^{\infty} $}} \Big(\frac 1{4k^2}\Big)^{{(1+\frac{\bar a \sigma_p}{2})^l}\cdot \frac{\bar a}{2}}   \\
  & \le \frac {C_1}{k^{\bar a}}\ \ \ \text{at}\ (\hat \theta,\hat{\mathfrak r},\hat t).
  \end{split}}
\end{equation}

Because $(\hat \theta,\mathfrak r,\hat t)$ is an arbitrary point near $(0, 0,T)$ with $\hat {\mathfrak r}>0$. Hence from \eqref{est1}
we conclude that (recall that $V=\zeta_{\mathfrak r}=2\frac{\zeta_s(\theta, s,t)}{s}$)
\begin{equation}\label{zss}
\lim_{\theta\to 0,  s\to 0^+,t\rightarrow T^-}\frac{\zeta_s(\theta, s,t)}{s}=\frac 12\lim_{\theta\to 0,  s\to 0^+,t\rightarrow T^-}V(\theta, s,t)=c_{nn}.
\end{equation}

The convergence \eqref{zss} implies that the constant $c_{nn}$ in the blow-up limit \eqref{zeta12}
is independent of the choice of the blow-up sequence.
Hence by the blow-up argument in the proof of Lemma \ref{lemconti1}, we infer that
\begin{equation}\label{ss00}
\lim_{\theta\to 0,  s\to 0^+,t\rightarrow T^-}\zeta_{ss}(\theta, s,t)=c_{nn}.
\end{equation}
By the convergence \eqref{ss00},
we can define $\zeta_{ss}$ on $\mathbb S^{n-1}\times\{s=0\}\times \{t=T\}$
as the limit $\lim_{ s\to 0^+}\zeta_{ss}(\theta, s,T)$.
The above proof also implies that $\zeta_{ss}$ is continuous on $\{s=0\}$.
For if not, let us assume that $\zeta_{ss}$ is dis-continuous at $(\theta, s,t)=(0,0,T)$.
Then there exist two sequences $(\theta^k_1, s^k_1,t^k_1)\to (0,0,T)$ and $(\theta^k_2, s^k_2,t^k_2)\to (0,0,T)$
such that $\zeta_{ss}(\theta^k_1, s^k_1,t^k_1)$ and $\zeta_{ss}(\theta^k_2, s^k_2,t^k_2)$
converge to different limits,
which is in contradiction with \eqref{ss00}.
This completes the proof.
\end{proof}

By a similar argument, we have

\begin{lemma}\label{lemconti-t}
Let $\zeta(\theta,s,t)$ be as in Theorem \ref{thm2}. Then $\zeta_{t}\in C(\mathbb S^{n-1}\times[0,1]\times (0,T])$.
\end{lemma}

\begin{proof}
As in Lemma \ref{lemconti2}, we introduce the coordinates $(\varphi,\tau,\sigma)$ 
and function $\tilde \zeta^k$  satisfying \eqref{zeta12}. 
By the convergence \eqref{zeta12} and the interior regularity of equation \eqref{po3},
we can choose a subsequence, such that
\begin{equation}\label{zeta-ct}
\Big|-\tilde \zeta^k_\sigma(\varphi,\tau,\sigma)-c_{0}\Big| \le \frac 1{k}
  \ \   \text{in}\  \big\{(\varphi,\tau,\sigma)\ | \  |\varphi|\le 1,\	\frac 1k\le \tau\le 1, -1\le \sigma\le 0\big\}.
\end{equation}
Scaling back to $\zeta(\theta,s,t)$, we obtain
\begin{equation}\label{aym-ct1}
\Big|-\zeta_t(\theta,s,t)-c_{0}\Big| \le \frac 1{k}
  \ \  \ \text{in}\  \big\{(\theta,s,t) \ |\ |\theta|\le \lambda_k,\	\frac {\lambda_k}k\le s\le \lambda_k, -\lambda_k^2\le t-T\le 0 \big\}.
\end{equation}
Differentiating equation \eqref{po2} with respect to $t$ and taking $V=\zeta_t$, one gets
\begin{equation}\label{eqn-ct2}
\begin{split}
\mathcal L(V)=: 
   & - V_t + \tilde a^{nn}\left(V_{ss}+  (1+1/\sigma_p ) \frac{V_{s}}{s} \right)+  {\Small\text{$ \sum_{i,j=1}^{n-1} $}}\tilde a^{ij}V_{\theta_i\theta_j}
      +  {\Small\text{$ \sum_{i=1}^{n-1}$}} \tilde a^{ni} V_{s\theta_i}
 =\tilde h
 \end{split}
\end{equation}
where $\tilde a^{ij},\tilde h$ are all bounded functions and $(\tilde a^{ij})_{i,j=1}^n$ is uniformly elliptic.
Then following the proof of  Lemma \ref{lemconti2} yields the present lemma.
\end{proof}

We also have the following lemma.

\begin{lemma}\label{lemconti3}
Let $\zeta(\theta,s,t)$ be as in Theorem \ref{thm2}.
Then $\zeta_{\theta\theta}\in C(\mathbb S^{n-1}\times[0,1]\times (0,T])$.
\end{lemma}
\begin{proof}
To prove the continuity of $  D_{\theta}^2\zeta$ on $\mathbb S^{n-1}\times\{s=0\}\times (0,T]$,  
it is enough to show $\lim_{s\rightarrow 0^+, \theta\to 0,t\rightarrow T^-}   D_{\theta}^2\zeta (\theta,s,t)$ exists.
By Lemma \ref{unif-ellip}, $   D_{\theta}^2 \zeta (\theta, s,t)$ is uniformly bounded.
Hence there is a sub-sequence  $s^k\rightarrow 0^+$ such that
$  D_{\theta}^2 \zeta (0, s^k,T)$ is convergent.
We introduce the coordinates $(\varphi,\tau,\sigma)$ and function $\tilde \zeta^k$
as in \eqref{coortran1} and \eqref{coortran2}, with $\theta^k=0$, $t^k=T$.

Let $\mathbb B$ denote the set of all convergent blow up sequences $\{\tilde \zeta^k\}$ 
given by \eqref{coortran2} (with $\theta^k=0$, $t^k=T$).
For any fixed unit vector $\nu\in \mathbb R^{n-1}$, define
\begin{equation}\label{blow-inf}
c_{\nu\nu}=\inf_{\{\tilde \zeta^k\} \in \mathbb B}\lim_{k\rightarrow +\infty}\tilde \zeta^k_{\nu\nu}(0',1,0)
\end{equation}
where $\tilde \zeta^k_{\nu\nu}=\tilde \zeta^k_{\theta_i\theta_j}\nu_i\nu_j$.
By a diagonal process, we can extract a subsequence in $\mathbb B$,
which for simplicity we still denote as  $\{\tilde \zeta^k\} $, such that
\begin{equation}
c_{\nu\nu}=\lim_{k\rightarrow +\infty} \tilde \zeta^k_{\nu\nu}(0',1,0).
\end{equation}\label{cgg}
We {\it claim}
\begin{equation}\label{add2}
 \varlimsup_{\theta\to 0, s\to 0^+,t\rightarrow T^-} \tilde \zeta_{\nu\nu}(\theta, s,t) \le c_{\nu\nu}.
\end{equation}

Indeed,
by the convergence \eqref{zeta12} and the interior regularity of equation \eqref{po3},
similarly to \eqref{zeta-c00} we can pass to a subsequence such that
\begin{equation*}
\big\| \tilde \zeta^k_{\nu\nu}(\varphi,\tau,\sigma)-c_{\nu\nu}\big\|_{L^\infty(Q_k)}\le \frac 1k\ \   \text{in}\  Q_k.
\end{equation*}
Scaling back to $\zeta(\theta,s,t)$, this implies
\begin{equation}\label{sigmak}
\big\|\tilde \zeta_{\nu\nu}(\theta,s,t)-c_{\nu\nu}\big\|_{L^\infty(\Sigma_k)}\le \frac 1k \
 \ \text{in}\ \Sigma_k.
\end{equation}
Here the domains $Q_k, \Sigma_k$  are the same as in
\eqref{zeta-c00} and \eqref{aym-aa2}.

To simplify the notation, let us denote the matrix in \eqref{po2} as $R=(r_{ij})_{i,j=1}^n$,
and rewrite equation \eqref{po2} as
\begin{equation}\label{po2-v1}
\mathcal F(\zeta_t,r_{ij})=: \log(-\zeta_t)+\log (\det R)=\log \bar F(s).
\end{equation}
Then  $\mathcal F$ is concave in its variables $r_{ij}$.
Differentiating \eqref{po2-v1} in direction $\nu$ twice and by the concavity, we have
\begin{equation*}
\frac{\zeta_{t,\nu\nu}}{\zeta_t}+\mathcal F_{r_{ij}}r_{ij,\nu\nu}\ge 0.
\end{equation*}
Denote $V=\zeta_{\nu\nu}$.
Similarly to \eqref{eqn-ct2}, one obtains
\begin{equation}\label{linear-1}
\begin{split}
\mathcal L(V)=: 
    & - V_t + \tilde a^{nn}\left(V_{ss}+  (1+1/\sigma_p ) \frac{V_{s}}{s} \right)
                     +  {\Small\text{$ \sum_{i,j=1}^{n-1} $}}\tilde a^{ij}V_{\theta_i\theta_j}
                      +  {\Small\text{$ \sum_{i=1}^{n-1}$}} \tilde a^{ni} V_{s\theta_i}
 \ge \tilde h ,
\end{split}
\end{equation}
where $\tilde a^{ij},\tilde h$ are all bounded functions and $(\tilde a^{ij})_{i,j=1}^n$ is uniformly elliptic.
Then following the proof of  Lemma \ref{lemconti2} yields \eqref{add2}.

To prove the convergence $ \lim_{\theta\to 0, s\to 0^+,t\rightarrow T^-} \zeta_{\nu\nu}(\theta, s,t)=c_{\nu\nu}$,
we make use of the equation \eqref{po2}.
By Lemmas \ref{lemconti1}-\ref{lemconti-t}, and noting that $ s \zeta_s = o(1)$ near $s=0$,
we can write \eqref{po2} as
\begin{equation}\label{add1}
\det( D^2_{\theta} \zeta+\zeta I)
=\frac{4}{\sigma_p^2}\left(\frac{1}{2c_0(1+1/\sigma_p)c_{nn}}\right)^{\frac 1p}+o(1) \ \text{ for }  (\theta,s,t)\text{ near }(0',0,T).
\end{equation}
The present lemma follows from the same argument as in Lemma 4.5 of \cite{HTW}.
\end{proof}

By Lemmas \ref{lemconti1} - \ref{lemconti3}, Theorem \ref{thm2} follows.

\vspace{5mm}
\section{Higher regularity for $\zeta$}\label{s6}
\vspace{3mm}

\subsection{\bf Regularity for linear parabolic equations}\label{cw-e}

Here we quote the $C^{2,\alpha}$ and $W^{2,p}$ estimates
for degenerate and singular  linear parabolic equations which will be needed later.

Given a point $p_0=(x_0, t_0)=(x'_0,x_{0,n},t_0)\in \mathbb R^{n,+} \times \R$ , denote
\begin{equation}\label{cylinder}
Q^*_\rho(p_0)=\{(x, t)\ |\
 x_n > 0, |x'-x'_0|< \rho, |x_n-x_{0,n}| < \rho^2, t_0-\rho^2 < t\le t_0\} ,
\end{equation}
which is a cylinder in $\mathbb R^{n,+} \times \R$.
When $p_0=(0, 0)$, we simply write  $Q^*_\rho=Q^*_\rho(p_0)$.

We first study the following linear degenerate operator
\begin{equation}\label{l++}
L_+ U=: - U_t + a_{nn}x_n\p_{nn} U
    + {\Small\text{$\sum_{i=1}^{n-1}$}} 2a_{in} \sqrt{x_n} \p_{in} U 
    + {\Small\text{$\sum_{i,j=1}^{n-1}$}} a_{ij} \p_{ij} U 
    +{\Small\text{$\sum_{i=1}^n$}}b_i \partial_i U      %=f   \quad   \text{in}\ Q^*_\rho.
\end{equation}
with variable coefficients $a_{ij}, b_i$ defined in the cylinder $Q^*_\rho$.

\begin{theorem}\label{T5.4}(Schauder estimate \cite{DH1999}).
Assume that the coefficients $a_{ij}, b_i\in C_\mu^{\alpha}(\overline{Q^*_\rho})$
for some $\alpha\in(0,1)$ and satisfy
\begin{equation}
{\begin{split}
  & a_{ij}\xi_i\xi_j\ge \Lambda^{-1} |\xi|^2  \ \ \  \forall\ \xi\in\mathbb R^n, \\
  &  \|a_{ij}\|_{C_\mu^{\alpha}(\overline{Q^*_\rho})}, \  \ \|b_{i}\|_{C_\mu^{\alpha}(\overline{Q^*_\rho})} \le \Lambda,
  \end{split}}
\end{equation}
and
\begin{equation}
b_n\ge \Lambda^{-1}  \ \   \text{at} ~~\, \{x_n=0\}
\end{equation}
for some positive constant $\Lambda$. Then for any given $\rho'\in(0,\rho)$,
there exists a constant $C$ depending only on $n, \alpha, \rho, \rho'$ and $\Lambda$, such that
\begin{equation}
 \|U\|_{C_\mu^{2+\alpha}(\overline{Q^*_{\rho'}})} 
       \le C \Big(\|U\|_{L^\infty(\overline{Q^*_\rho})} + \|L_+ U\|_{C_\mu^{\alpha}(\overline{Q^*_\rho})}\Big),
\end{equation}
for all functions $U\in C_\mu^{2+\alpha}(\overline{Q^*_\rho})$.
\end{theorem}

We also need a local $W^{2,p}$ estimate for the following singular linear parabolic equation
\begin{equation}\label{new1}
  -U_t + {\Small\text{$\sum_{i,j=1}^{n}$}} a^{ij}\partial_{ij} U
        + {\Small\text{$\sum_{i=1}^{n-1}$}}b^i\partial_i U+\frac{b_n}{x_n}\partial_n U+cU=f  \quad   \text{in}\ Q_\rho.
\end{equation}
Here $Q_\rho$ is the cylinder defined in \eqref{box-r}.

By  \cite[Theorem 2.7]{DP2020},  we have the following a priori estimates.

\begin{theorem}\label{T5.3}
Let $U\in  W^{2,1}_{p}( Q_\rho,d\nu)$
be a solution to  \eqref{new1} with $f\in L^p_{\nu} (Q_\rho)$ for some $p>1$.
Assume that $a^{ij}, b^i, c\in C^0(\overline {Q_\rho})$ satisfy conditions \begin{equation}\label{new2}
\begin{split}
  &\Lambda^{-1} I_{n\times n} \le (a^{ij})_{i,j=1}^n  \le \Lambda I_{n\times n} \ \ \ \text{in} \ \overline
{Q_\rho},\\
& \frac{b^n}{a^{nn}}= b> 1 \ \ \text{is~a~constant}, \\
&|c|+{\Small\text{$\sum_{i=1}^{n}$}} |b^i|\le \Lambda\ \ \ \text{in} \ \overline
{Q_\rho},
\end{split}
\end{equation}
for some positive constant $\Lambda$, and
\begin{equation}\label{un0}
\lim_{x_n\to 0^+} x_n^b U_n(x',x_n,t)=0  \ \   \text{in} ~~ \overline{Q_\rho}.
\end{equation}
Then for any $\rho'\in(0,\rho)$, $U$ satisfies the estimate
\begin{equation}\label{est-fullp}
\|U\|_{W^{2,1}_{p}(Q_{\rho'},d\nu)}+\Big\|\frac{U_n}{x_n} \Big\|_{L^p_{\nu}(Q_{\rho'})}
      \le C\left(\|f\|_{L^p_{\nu}(Q_\rho)}+\|U\|_{L^{p}_{\nu}(Q_\rho)} \right),
%+ \|Du\|_{L^{p}_{\mu_b}(Q^*^+_1)}\right) ,
\end{equation}
where  $C>0$ depends only on $p, n, \Lambda,b,\rho,\rho'$ and the modulus of continuity of $a^{ij}, b^i$ and $c$.

\end{theorem}

\vskip10pt

\subsection{\bf Higher regularity for $\zeta$}

By the $C^{2,\alpha}$ and $W^{2,p}$ estimates in Section \ref{cw-e},
we can prove higher order regularity for the function $\zeta$ defined in \eqref{zeta0}.

\begin{theorem}\label{thm4}
Let $\zeta(\theta,s,t)\in C^2 (\mathbb S^{n-1}\times [0,1]\times (0,T])$ be a solution to \eqref{po2}.
Assume that $\zeta_{s}(\theta,0,t)=0$\ $\forall~\theta\in \mathbb S^{n-1}$ and $t\in (0,T]$.
Then for $\tau>0$
\begin{equation} \label{festi}
\|\zeta\|_{ C^{2+\alpha}(\mathbb S^{n-1}\times [0,1]\times [\tau,T])}\le C(\mathcal M_0,n,p,\tau,T), 
\end{equation}
for some  $\alpha\in (0,1)$.
Moreover,
\begin{equation}\label{high-est-a}
\begin{split}
& \|D^{k}_{\theta,t}D_s^l\zeta\|_{L^\infty(\mathbb S^{n-1}\times [0,1]\times [\tau,T])}+\|D^{k}_{\theta,t}(\zeta_s/s)\|_{L^\infty(\mathbb S^{n-1}\times [0,1]\times [\tau,T])}\\
\le & C(\mathcal M_0,n,p,\tau,T,k), \ \forall\   k\in \mathbb N,\quad l=0,1,2.
\end{split}
\end{equation}
\end{theorem}

\begin{proof}
Differentiating \eqref{po2} with respect to $\theta_k$, $k=1,\cdots,n-1$, one gets
\begin{equation}\label{high1}
\begin{split}
\mathcal L_0(V)=:\  - V_t +  a^{nn}\Big( V_{ss}+  {\small\text{$  \frac{2+\sigma_p}{\sigma_p} $}} \frac{V_{s}}s \Big)
       +  {\Small\text{$ \sum_{i,j=1}^{n-1} $}} a^{ij}V_{\theta_i\theta_j} 
       +  {\Small\text{$ \sum_{i=1}^{n-1}$}}  a^{ni} V_{s\theta_i}  + b^n V_{s}=\bar h,
\end{split}
\end{equation}
where $V=\zeta_{\theta_k}$ and $a^{ij}$, $b^i$, $\bar h$ are
continuous functions of $s,\zeta, \zeta_t, s\zeta_s, \frac{\zeta_s}s, D^2\zeta$.
To apply the a priori estimates in Section \ref{cw-e} to equation \eqref{high1}, 
we express the equation in a {\it local coordinates} on $\mathbb S^{n-1}$.
By Theorem \ref{thm2},  $a^{ij}$, $\bar h$ and $V_s$ are continuous in $\theta,s$ and $t$.
By Lemma \ref{unif-ellip}, the operator $\mathcal L_0$ is uniformly parabolic.
Hence all the assumptions in Theorem \ref{T5.3} are fulfilled for $V$  with
$b=\frac{2+\sigma_p}{\sigma_p}>1$.
Hence we obtain
\begin{equation*}
\|\zeta_\theta\|_{W^{2,1}_{q}(\mathbb S^{n-1}\times [0,1]\times [\tau,T],d\nu)}\le C(\mathcal M_0,n,p,\tau,T,q), \ \   \forall\ q>1.
\end{equation*}

Similarly, differentiating \eqref{po2} with respect to $t$,
we obtain   \eqref{high1} for  $ V=\zeta_t$.
By Theorem \ref{T5.3}, we then obtain
\begin{equation*}
\|\zeta_t\|_{W^{2,1}_{q}(\mathbb S^{n-1}\times [0,1]\times[\tau,T],d\nu)}\le C(\mathcal M_0,n,p,\tau,T,q), \ \   \forall \ q>1.
\end{equation*}
Thus $\|D_{\theta,s}\zeta_\theta\|_{W^{1,1}_q (\mathbb S^{n-1}\times [0,1]\times[\tau,T],d\nu)}\le C$
and $\|\zeta_t\|_{ W^{1,1}_{q}(\mathbb S^{n-1}\times [0,1]\times[\tau,T],d\nu)}\le C$.
Letting $q>n+1+b$ and by the Sobolev embedding,
$W^{1,1}_{q}(d\nu)\to C^\alpha$ \cite[Lemma 4.65 and Lemma 4.66] {AdamsFournier2003},
we have $ D_{\theta,s} \zeta_{\theta} , \zeta_t \in C^\alpha(\mathbb S^{n-1}\times [0,1]\times (0,T])$.

Write equation \eqref{po2} in the form
\begin{equation}\label{ode1}
\zeta_{ss}+\frac{2+\sigma_p}{\sigma_p}\frac{\zeta_s}{s}=\tilde f, %\in C^\alpha(\mathbb S^{n-1}\times [0,1]\times [0,T]) ,
\end{equation}
where $\tilde f$ is a H\"older function of all its arguments  $s,\zeta, \zeta_t, D_{\theta,s} \zeta,  D_{\theta,s} \zeta_\theta$. 
Hence $\tilde f$ is H\"older continuous in $\theta, s, t$.
The solution to \eqref{ode1} is given by
\begin{equation}\label{ode2}
\zeta(\theta, s,t)=\zeta(\theta, 0,t)+\int_0^s  r^{-\frac{2+\sigma_p}{\sigma_p}}\int_0^r \lambda^{\frac{2+\sigma_p}{\sigma_p}} \tilde f(\theta,\lambda,t)d\lambda dr.
\end{equation}
Hence we have
\begin{equation*}\label{ode3}
\begin{split}
 \zeta_s(\theta,s,t)
    &=s^{-\frac{2+\sigma_p}{\sigma_p}}\int_0^s \lambda^{\frac{2+\sigma_p}{\sigma_p}} \tilde f(\theta,\lambda,t)d\lambda,\\
 \zeta_{ss}(\theta,s,t)
    &=-\frac{2+\sigma_p}{\sigma_p}s^{-\frac{2+\sigma_p}{\sigma_p}-1}
         \int_0^s \lambda^{\frac{2+\sigma_p}{\sigma_p}} \tilde f(\theta,\lambda,t)d\lambda  +\tilde f.
                                 %\in C^\alpha(\mathbb S^{n-1}\times[0,1]\times[0,T]).
\end{split}
\end{equation*}
This implies $\zeta \in C^{2+\alpha}(\mathbb S^{n-1}\times [0,1]\times (0,T]) $
and $\frac{\zeta_s}{s} \in C^\alpha (\mathbb S^{n-1}\times [0,1]\times (0,T])$.
Recall that $\mathfrak r=\frac {s^2}4$,
we obtain $\zeta(\theta,\mathfrak r,t)\in C_\mu^{2+\alpha}(\mathbb S^{n-1}\times [0,1/4]\times (0,T])$.

Differentiating equation \eqref{po2} with respect to $\theta$ and $t$ again,
we obtain $D_{\theta,t}^k\zeta(\theta,\mathfrak r,t)\in C^{2+\alpha}_\mu(\mathbb S^{n-1}\times [0,1/4]\times (0,T])$,
by the Schauder estimate  in Section \ref{cw-e}. This also proves estiamates \eqref{high-est-a}. 
\end{proof}

 \begin{remark}
 The smoothness of the interface $\Gamma_t$  follows from the higher regularity of $\zeta$ in Theorem \ref{thm4}. Indeed,
 one can define the section $S_{1,\phi,t} =: \{x\in\R^n \ |\  \phi(x,t)<1\}$, which is the polar body of $\{y: v(y,t)=0\}$,  i.e.,
\begin{equation}\label{dual}
S_{1, \phi,t} =\big\{x\in\mathbb R^n \ |\  x\cdot y< 1\ \ \forall\ y\in \{y\	|\	v(y,t)=0\}\big\}.
\end{equation}
Hence $\p S_{1, \phi,t}$ is $C^k$ smooth ($k\ge 2$) and uniformly convex if and only if  the interface $\Gamma_t$ is.
\end{remark}
 
\begin{remark}\label{rem6.2}
Theorem \ref{thm4} also implies the conditions {\rm (I1)-(I4) } for $t\in (0,T^*)$.
Indeed,   {\rm (I1)} and {\rm (I2) } for $t\in (0,T^*)$ are easy to verify. 
The verification of $g(\cdot, t)\in C_\mu^{2+\alpha}(\overline{\{v>0\}})$ and $g_{ij}\tau_i g_j(\cdot,t)\in L^\infty$ need more computation. The calculations also can prove the  regularity of $g$ up to the interface $\Gamma_t$.
\begin{itemize}
\item[$(1)$] if $\frac{2}{\sigma_p}\in \mathbb Z^+$,   the function $g=\big(\frac{\sigma_p+1}{\sigma_p} v\big)^{\frac{\sigma_p}{\sigma_p+1}}$ is  smooth up to the interface $\Gamma_t$ on $0< t <T^*$;
\item[$(2)$] if $\frac{2}{\sigma_p}\notin \mathbb Z^+$, the function $g$ is $C_\mu^{\big[\frac{2}{\sigma_p}\big], 2+\frac{2}{\sigma_p} - \left[\frac{2}{\sigma_p}\right]}$ up to the interface $\Gamma_t$ on $0< t <T^*$.
\end{itemize}
The  proof of the  regularity for $g$ is cumbersome. Hence, we will present the details of the proof in a  future  work \cite{HWZ}. 
\end{remark}

By the a priori estimate \eqref{festi} and the continuity method \cite{L1996}, 
we obtain the existence of smooth solutions to equation \eqref{po2}.
By Remark 6.1, it implies the smoothness of the interface $\Gamma_t$,
and thus completes the proof of Theorem \ref{thmA}.

\end{document}